\documentclass{aims}
\usepackage{amsmath}
  \usepackage{paralist}
  \usepackage{graphics} 
  \usepackage{epsfig} 
\usepackage{graphicx}  \usepackage{epstopdf}
 \usepackage[colorlinks=true]{hyperref}
\usepackage{amssymb}
\usepackage{amsfonts}
\usepackage{graphicx}
\usepackage[abs]{overpic}
\usepackage{caption}
\usepackage{wrapfig}
\usepackage{float}
\usepackage{graphicx}
\usepackage{mathrsfs}
\usepackage{amsthm}
\usepackage{caption}
\usepackage{subcaption}

\usepackage{tikz-cd}
\usepackage[colorinlistoftodos]{todonotes}
 
\hypersetup{urlcolor=blue, citecolor=red}

\def\currentvolume{X}
 \def\currentissue{X}
  \def\currentyear{200X}
   \def\currentmonth{XX}
    \def\ppages{X--XX}
     \def\DOI{10.3934/xx.xx.xx.xx}

\newtheorem{theorem}{Theorem}[section]

\newtheorem{lemma}[theorem]{Lemma}

\theoremstyle{definition}
\newtheorem{definition}[theorem]{Definition}
\newtheorem{remark}{Remark}

\newcommand{\eps}[1]{{#1}_{\varepsilon}}

\newcommand{\N}{\Bbb N}

\newcommand{\R}{\Bbb R}
\newcommand{\C}{\Bbb C}

\newcommand{\stress}{{T^E}}
\newcommand{\vstress}{{S}}

\def\Id{\mathbf{Id}}

\def\eps{\varepsilon}

\def\dist{\operatorname{dist}}

\def\XXint#1#2#3{{\setbox0=\hbox{$#1{#2#3}{\int}$}
     \vcenter{\hbox{$#2#3$}}\kern-.5\wd0}}

\usepackage{color}

\newcommand{\cm}{\color{black}}

\newcommand{\BBB}{\color{black}} 
\newcommand{\EEE}{\color{black}}

\newcommand{\PPP}{\color{black}} 
 
\title[Numerical approximation of viscoelastic plates] 
      {Numerical approximation of  von K\'{a}rm\'{a}n viscoelastic plates}

\author[Manuel Friedrich, Martin Kru\v{z}\'{\i}k and Jan Valdman]{}


 \email{manuel.friedrich@uni-muenster.de}
 \email{kruzik@utia.cas.cz}
 \email{jan.valdman@utia.cas.cz}

\thanks{$^*$ Corresponding author: Martin Kru\v{z}\'{\i}k}

\begin{document}
\maketitle

\centerline{\scshape Manuel Friedrich}
\medskip
{\footnotesize
 \centerline{Institute for Computational and Applied Mathematics}
   \centerline{University of  M\"{u}nster}
   \centerline{Einsteinstr.~62, D-48149 M\"{u}nster, Germany}
} 

\medskip

\centerline{\scshape Martin Kru\v{z}\'ik$^*$}
\medskip
{\footnotesize
 \centerline{Czech Academy of Sciences, Institute of Information Theory and Automation}
   \centerline{Pod vod\'arenskou v\v{e}\v{z}\'i 4, CZ-182 08 Praha 8, Czechia}
   \centerline{Faculty of Civil Engineering, Czech Technical University}
	  \centerline{Th\'akurova 7, CZ-166 29 Praha 6, Czechia}  
}

\medskip

\centerline{\scshape Jan Valdman}
\medskip
{\footnotesize
 \centerline{Czech Academy of Sciences, Institute of Information Theory and Automation}
   \centerline{Pod vod\'arenskou v\v{e}\v{z}\'i 4, CZ-182 08 Praha 8, Czechia}
   \centerline{Institute of Mathematics,  University of South Bohemia}
	  \centerline{Brani\v{s}ovsk\'{a} 1760, CZ-370 05 \v{C}esk\'{e} Bud\v{e}jovice, Czechia}  
}

\bigskip


\bigskip

\centerline{{\it This paper is dedicated to Alexander Mielke in the occasion of his 60th birthday.}}

\begin{abstract}
We consider metric gradient flows and their discretizations in time and space. \BBB We prove an abstract  convergence result for time-space discretizations and identify their limits as curves of maximal slope. As an application, we consider \EEE  a finite element approximation of a quasistatic evolution \BBB for \EEE viscoelastic von K\'{a}rm\'{a}n plates \cite{VonKarman}. Computational experiments are provided, too.    
\end{abstract}

\section{Introduction}
Neglecting inertia,  a   nonlinear viscoelastic material in Kelvin's-Voigt's rheology (i.e., a spring and a dashpot coupled in parallel)  obeys the following system of equations 
\begin{align}\label{eq:viscoel}
-{\rm div}\Big(\partial_FW(\nabla y)   + \partial_{\dot F}R(\nabla y,\partial_t \nabla y)  \Big) =  f\text{ in $[0,T]\times \Omega$.} \end{align}
Here, $[0,T]$ is the process time interval with $T>0$,   $\Omega\subset\R^3$  is a smooth bounded domain representing the reference configuration,  and  $y:[0,T]\times \Omega\to\R^3$ is the deformation mapping  with corresponding   deformation gradient $\nabla y$.  Further, \EEE $W:\R^{3\times 3}\to  [0,\infty]\EEE$ is a stored energy density, which represents a potential of the first Piola-Kirchhoff stress tensor $\stress$, i.e.,  $\stress:=\partial_F W$, and  $F\in\R^{3\times 3}$ is the placeholder for $\nabla y$. Moreover, $R:  \R^{3 \times 3} \times \R^{3 \times 3} \to [0,\infty) \EEE $ denotes a (pseudo)potential of dissipative forces,  where $\dot F \in \R^{3 \times 3}$ is the placeholder of $\partial_t \nabla y$. Finally,  \EEE    $f:\Omega\to\R^3$ is a volume density of  external forces acting on $\Omega$.  

A standard assumption for $W$ is frame indifference, i.e.,  $W(F)=W(QF)$   for every proper rotation $Q\in{\rm SO}(3)$ and every $F\in\R^{3\times 3}$. This  implies that $W$ depends on the right Cauchy-Green strain tensor  $C:=F^\top F$, see e.g.~\cite{Ciarlet}. The second term on the left-hand side of \eqref{eq:viscoel} is the stress tensor 
$\vstress(F,\dot F):= \partial_{\dot F} R(F,\dot F)$ which has its origin in viscous dissipative mechanisms of the material.  Notice that its potential $R$ plays an  analogous role as $W$ in the case of purely elastic, i.e., non-dissipative processes. Naturally, we require that $R(F,\dot F)\ge R(F,0)=0$.  The viscous stress tensor must comply with the time-continuous frame-indifference principle,  meaning that $\vstress(F,\dot F)=F\tilde\vstress(C,\dot C)$, where $\tilde\vstress$ is a symmetric matrix-valued function.   This condition constraints 
$R$ so that \cite{Antmann, Antmann:04,MOS}
\begin{align*}
R(F,\dot F)=\tilde R(C,\dot C)
\end{align*}
for some nonnegative function $\tilde R$.  In other words, \EEE $R$ must depend on the right Cauchy-Green strain tensor $C$ and its time derivative $\dot C$.

\BBB Recently,  in \cite{MFMK}, the first two authors proved the existence of weak solutions to equations of  the form  \eqref{eq:viscoel}   in three-dimensional nonlinear viscoelasticity for nonsimple materials. While the elastic properties of simple  elastic materials depend only on the first gradient, the notion of a  nonsimple (or second-grade) material \EEE refers to the fact that the elastic energy additionally depends  on the second gradient of the deformation. This concept, pioneered by  Toupin \cite{Toupin:62,Toupin:64},  has proved to be useful in modern mathematical elasticity, see e.g.~\cite{BCO,Batra, capriz,dunn,MR,MielkeRoubicek:16,Podio}. \BBB Adopting this setting currently appears  to be inevitable to establish the existence of solutions, see \EEE \cite{MFMK},  and  \cite{MOS} for a general discussion about the interplay between the elastic energy and viscous dissipation. \BBB We emphasize, however, that a main justification of the investigated model  is the observation   that, in the small strain limit, the problem leads to the standard system  of linear   viscoelasticity without second gradient.  

In the present paper, we are interested in the analysis of lower-dimensional analogs of \eqref{eq:viscoel} which are derived  by considering  \eqref{eq:viscoel} for thin viscoelastic plates and  by passing to the vanishing-thickness limit. Such studies, often referred to as \emph{dimension  reduction}, \EEE  play a significant  role in nonlinear analysis and numerics since they allow for simpler  computational approaches still  preserving main features of the full-dimensional system. 
In particular, it is important that  the relationship between the \BBB original models   and their lower-dimensional counterparts \EEE  is  made rigorous.  Usually, \BBB the main tools in a \EEE    variational setting are $\Gamma$-convergence \cite{DalMaso:93} \BBB and geometric \EEE rigidity estimates \cite{FrieseckeJamesMueller:02}.  We refer to \cite{Dret1,Dret2} for a derivation of membrane models     from three-dimensional elasticity or to \cite{Casarino,FrieseckeJamesMueller:02, hierarchy, Pantz} for  analogous \BBB approaches to plate theory. \EEE   

\BBB In the framework of nonsimple viscoelastic materials, such a  scenario was recently studied by the first two authors  in \cite{Friedrich-Kruzik-19}, where  a von K\'{a}rm\'{a}n-like  viscoelastic plate model has been identified as an effective 2D  dimension-reduction limit.    \BBB For this analysis, \EEE besides  rigidity estimates and $\Gamma$-convergence,  the main tools are gradient flows in metric spaces developed in \cite{AGS, Ortner,S1,S2}.  Although there are previous works on viscoelastic plates \cite{Bock1,Park}, some even including inertial effects  \cite{Bock2, Bock3}, their starting point is already a plate model. \BBB In contrast, \EEE \cite{Friedrich-Kruzik-19} provides a rigorous derivation from a three-dimensional model of viscoelasticity at finite strains \BBB by (i) showing the existence of solutions to the effective 2D system, and by (ii) by proving that these solutions are in a certain sense the limits of solutions to the 3D equations for vanishing thickness. \EEE

 The main aim of this contribution is to carry out a finite-element convergence  analysis of a fully discrete viscoelastic plate model and to \BBB investigate \EEE its behavior by computational experiments. \BBB As a byproduct, we also obtain an alternative existence proof for solutions to the effective 2D system. This analysis is based on proving an abstract  convergence result of time-space discretizations  to metric gradient flows, see Theorem~\ref{th:abstract convergence-numerics}. \EEE

\BBB  At many spots,  \EEE our   strategy relies   on results obtained in \cite{Friedrich-Kruzik-19} and on the theory of  gradient flows in metric spaces \cite{AGS} which provide us with  a robust approach to quasistatic evolutionary problems.  In particular, Theorem~\ref{th:abstract convergence-numerics} exploits a sequence of minimization problems to construct fully discrete approximations  (see \eqref{eq: ds-new1} \BBB and \eqref{eq: ds-new1NNN}) \EEE of curves of maximal slope which are then solutions to the \BBB viscoelastic \EEE plate equations.  
This makes the proof partially constructive and, at the same time, it suggests a numerical method to be used.

The plan of the paper is as follows. Section~\ref{sec:viscoel} reviews equations of nonlinear viscoelasticity \BBB in the framework of nonsimple materials \EEE and the resulting system for the von K\'{a}rm\'{a}n plates. Mathematical tools from the theory of gradient flows in metric spaces \cite{AGS}, such  as generalized minimizing movements  and curves of maximal slope \cm \cite{Am95,DGMT}, \EEE are introduced in Section~\ref{sec3}.  Moreover, Section~\ref{sec3}  contains our main abstract convergence result \BBB for time-space discretizations whose limits are curves of maximal slope, see Theorem \ref{th:abstract convergence-numerics}.  \EEE Section~\ref{sec:FEM} applies the  abstract results to \BBB  the 2D system of viscoelastic von K\'{a}rm\'{a}n plate equations, see \eqref{eq: equation-simp} below: we provide \EEE  an approximation of the original problem by a finite element method, see Theorem~\ref{maintheorem}. \BBB As a byproduct, this approximation result yields \EEE an alternative proof of the existence of \BBB solutions  to the \EEE viscoelastic plate model originally obtained in \cite{Friedrich-Kruzik-19}. Finally, Section~\ref{numexp} \BBB provides  \EEE computational examples  simulating the behavior of the \BBB viscoelastic  \EEE plate exposed to external \BBB forces. \EEE

We use standard notation for Lebesgue spaces,  $L^p(\Omega)$, which consist of  measurable maps on $\Omega\subset\R^d$, $d=2,3$, that are integrable with the $p$-th power (if $1\le p<+\infty$) or essentially bounded (if $p=+\infty$).    With $W^{k,p}(\Omega)$ we denote Sobolev spaces, i.e.,   linear spaces of  maps  which, together with their weak derivatives up to the order $k\in\N$, belong to $L^p(\Omega)$.  Further,  $W^{k,p}_0(\Omega)$ contains maps from $W^{k,p}(\Omega)$ having zero boundary conditions (in the sense of traces).  To emphasize the target space $\R^k$, $k=1,2,3$, we write $L^p(\Omega;\R^k)$. If $k=1$, we write $L^p(\Omega)$ as usual.   We refer to \cite{AdamsFournier:05} for more details on Sobolev spaces.  We  also denote the components of vector  functions $y$ by $y_1$, $y_2$, and  $y_3$, and so on. \BBB By $\Id$ we denote the identity matrix in $\R^{3 \times 3}$. \EEE  If $A\in\R^{3\times 3\times 3\times 3}$ and $e\in\R^{3\times 3}$, then $Ae\in\R^{3\times 3}$ is such that for $i,j\in\{1,2,3\}$ we define  $(Ae)_{ij}:=A_{ijkl}e_{kl}$ where we use Einstein's summation convention. An analogous  convention is used in similar  occasions, in the sequel.
Finally, at many spots, we closely follow the notation introduced in \cite{AGS} to ease readability of our work because the theory developed there is one of the main tools of our analysis.

\section{Equations of viscoelasticity in 3D and  2D }\label{sec:viscoel}
We first introduce a 3D setting following the setup in \cite{Friedrich-Kruzik-19, hierarchy}. We consider  a right-handed  orthonormal system $\{e_1,e_2,e_3\}$ \PPP and \EEE  $S \subset \R^2$ open, bounded with Lipschitz boundary, in the span of $e_1$ and $e_2$. \PPP  Let $h>0$ small. We consider  \EEE  \emph{deformations} $w: S \times (-\frac{h}{2},\frac{h}{2}) \to \R^3$. It is convenient to work in a fixed domain $\Omega = S \times I$ with $I:= (-\frac{1}{2},\frac{1}{2})$ and to rescale deformations according to $y(x) = w(x',hx_3)$, so that $y: \Omega \to \R^3$, where we use the abbreviation $x' = (x_1,x_2)$. We also introduce the notation $\nabla' y = y_{,1} \otimes e_1  +y_{,2} \otimes e_2$  for the in-plane gradient, and the scaled gradient
\begin{align}\label{eq: scaled1} 
\nabla_h y := \Big(\nabla' y, \frac{1}{h} y_{,3} \Big) = \nabla w.
\end{align}
Moreover, we define the scaled second gradient by 
\begin{align}\label{eq: scaled2} 
(\nabla^2_h y)_{ijk} := h^{-\delta_{3j} - \delta_{3k}} (\nabla^2 y)_{ijk} = (\nabla^2 w)_{ijk} =\partial^2_{jk} w_i \ \  \text{for $i,j,k \in \lbrace 1,2,3\rbrace$}, 
\end{align}
where $\delta_{3j},\delta_{3k}$ denotes the Kronecker delta. 

\smallskip

\textbf{Stored elastic energy  density  \EEE and body forces:} 
 We assume that $W: \R^{3 \times 3} \to [0,\infty]$ is a single-well, frame-indifferent stored energy  density \EEE with the usual assumptions in nonlinear elasticity. We   suppose  that there exists $c>0$ such that
\begin{align}\label{assumptions-W}
\begin{split}
(i)& \ \ W \text{ continuous and  $C^3$ \EEE in a neighborhood of $SO(3)$},\\
(ii)& \ \ \text{frame indifference: } W(QF) = W(F) \text{ for all } F \in \R^{3 \times 3}, Q \in SO(3),\\
(iii)& \ \ W(F) \ge c\dist^2(F,SO(3)), \  W(F) = 0 \text{ iff } F \in SO(3),
\end{split}
\end{align}
where $SO(3) = \lbrace Q\in \R^{3 \times 3}: Q^\top Q = \Id, \, \det Q=1 \rbrace$.  Moreover, for $p>3$, \EEE let $P: \R^{3\times 3 \times 3} \to [0,\infty]$ be a higher order perturbation satisfying 
\begin{align}\label{assumptions-P}
\begin{split}
(i)& \ \ \text{frame indifference: } P(QZ) = P(Z) \text{ for all } Z \in \R^{3 \times 3 \times 3}, Q \in SO(3),\\
(ii)& \ \ \text{$P$ is convex and $C^1$},\\
(iii)& \ \ \text{growth condition: For all $Z \in \R^{3 \times 3 \times 3}$ we have } \\&   \ \ \ \ \ \    c_1 |Z|^p \le P(Z) \le c_2 |Z|^p, \ \ \ \ \ \ |\partial_{Z} P(Z)|  \le c_2 |Z|^{p-1} 
\end{split}
\end{align}
for $0<c_1<c_2$. Finally, $f \in L^\infty(\Omega)$ denotes a volume normal force,   i.e., a force oriented in the  $e_3$ direction.    

\smallskip

\textbf{Dissipation potential and viscous stress:} We now introduce a dissipation potential. We follow here the discussion in \cite[Section 2.2]{MOS} and \cite[Section 2]{Friedrich-Kruzik-19}.  Consider a time-dependent deformation $y: [0,T] \times \Omega \to \R^3$. Viscosity is not only related to the strain rate $\partial_t \nabla_h  y(t,x)$  but also to the strain $\nabla_h y(t,x)$.  It  can be expressed in terms of a  dissipation potential $R(\nabla_h y, \partial_t \nabla_h y)$, where $R: \R^{3 \times 3} \times \R^{3 \times 3} \to [0,\infty)$. An admissible potential has to satisfy frame indifference in the sense (see \cite{Antmann, MOS})
\begin{align}\label{R: frame indiff}
R(F,\dot{F}) = R(QF,Q(\dot{F} + AF))  \ \ \  \forall  Q \in SO(3), A \in \R^{3 \times 3}_{\rm skew}
\end{align}
for all $F \in GL_+(3)$ and $\dot{F} \in \R^{3 \times 3}$, where $GL_+(3) = \lbrace F \in \R^{3 \times 3}: \det F>0 \rbrace$ and $\R^{3 \times 3}_{\rm skew} = \lbrace A  \in \R^{3 \times 3}: A=-A^\top \rbrace$. 

From the viewpoint of modeling,  it is  more  convenient to postulate the existence of a (smooth) global distance $D: GL_+(3) \times GL_+(3) \to [0,\infty)$ satisfying $D(F,F) = 0$ for all $F \in GL_+(3)$. From this, an associated dissipation potential $R$ can be calculated by
\begin{align}\label{intro:R}
R(F,\dot{F}) := \lim_{\eps \to 0} \frac{1}{2\eps^2} D^2(F+\eps\dot{F},F) = \frac{1}{4} \partial^2_{F_1^2} D^2(F,F) [\dot{F},\dot{F}]
\end{align}
for $F \in GL_+(3)$ and  $\dot{F} \in \R^{3 \times 3}$. Here, $\partial^2_{F_1^2} D^2(F_1,F_2)$ denotes the Hessian of  $D^2$ in the direction of $F_1$ at $(F_1,F_2)$, which is a fourth order tensor.  For some $c>0$ we suppose that $D$ satisfies 
\begin{align}\label{eq: assumptions-D}
(i) & \ \ D(F_1,F_2)> 0 \text{ if } F_1^\top F_1 \neq F_2^\top F_2,\notag \\
(ii) & \ \ D(F_1,F_2) = D(F_2,F_1),\\
(iii) & \ \ D(F_1,F_3) \le D(F_1,F_2) + D(F_2,F_3),\notag \\
(iv) & \ \ \text{$D(\cdot,\cdot)$ is $C^3$ in a neighborhood of $SO(3) \times SO(3)$},\notag 
\\
(v)& \ \ \text{Separate frame indifference: } D(Q_1F_1,Q_2F_2) = D(F_1,F_2)\notag \\
& \ \  \ \ \ \ \ \ \ \ \    \ \  \ \ \ \ \ \ \ \ \    \ \  \ \ \ \ \ \ \ \ \    \ \  \ \ \ \ \ \ \ \ \   \forall Q_1,Q_2 \in SO(3),  \ \forall F_1,F_2 \in GL_+(3),\notag\\ 
(vi) & \ \ \text{$D(F,\Id) \ge c\dist(F,SO(3))$  $\forall F \in \R^{3 \times 3}$ in a neighborhood of $SO(3)$}.\notag
\end{align}
Note that conditions (i)-(iii) state that $D$ is a true distance when restricted to symmetric matrices   with nonnegative determinants. We cannot expect more due to the separate frame indifference (v). We also point out that (v) implies \eqref{R: frame indiff} as shown in \cite[Lemma 2.1]{MOS}. Note that in our model we do not require any conditions of polyconvexity  \cite{Ball:77} neither for $W$ nor for $D$.  One possible example of $D$ satisfying 
\eqref{eq: assumptions-D} \BBB is \EEE  $D(F_1,F_2)=  |F_1^\top F_1-F_2^\top F_2|\EEE$.   This choice leads  to $R(F,\dot F)=| F^\top \dot F + \dot F^\top F \EEE |^2/2$.  For further examples  we refer  to \cite[Section 2.3]{MOS}.

\smallskip

\textbf{Equations of viscoelasticity in a rescaled  domain:} Following the study in \cite{Friedrich-Kruzik-19, lecumberry}, \BBB we introduce  \EEE the set of admissible configurations by
\begin{align}\label{eq: nonlinear boundary conditions}
\mathfrak{S}_h = \Big\{ y \in W^{2,p}(\Omega;\R^3): \ y(x',x_3) = \begin{pmatrix} x' \\ hx_3 \end{pmatrix} 
\text{for } x' \in \partial S, \ x_3 \in I \Big\},
\end{align} 
where $I=(-\frac{1}{2},\frac{1}{2})$. \BBB Note that in \cite{Friedrich-Kruzik-19, lecumberry} more general clamped boundary conditions are considered that are  not included here for the sake of simplicity.  \EEE       We formulate the equations of viscoelasticity  for a nonsimple material involving the perturbation $P$ (cf.\ \eqref{assumptions-P}). \EEE We introduce a differential operator associated with  $P$. To this end, we recall the notation of the scaled gradients in \eqref{eq: scaled1}-\eqref{eq: scaled2}. For $i,j \EEE \in \lbrace 1,2, 3\rbrace$, we denote by $(\partial_ZP(\nabla^2_h y))_{ij*}$ the vector-valued function  $((\partial_ZP(\nabla^2_h y))_{ijk})_{k=1,2,3}$. We also introduce the    scaled (distributional) divergence  ${\rm div}_h g$  for a function $g \in L^1(\Omega;\R^3)$  by \EEE ${\rm div}_h g = \partial_1 g_1+ \partial_2 g_2 + \frac{1}{h}\partial_3 g_3$. We define
\begin{align*}
\big(\mathcal{L}^h_P(\nabla^2_h y)\big)_{ij} =  - {\rm div}_h (\partial_ZP(\nabla^2_h y))_{ij*}, \ \ \ \  i,j \EEE \in \lbrace 1,2, 3\rbrace
\end{align*}   
for $y \in \mathfrak{S}_h$. \BBB Let $0 < \beta < 4$. \EEE  The equations of nonlinear viscoelasticity  \BBB are defined by \EEE
\begin{align}\label{nonlinear equation}
\begin{cases} -  {\rm div}_h \Big( \partial_FW(\nabla_h y) +   h^{\beta} \EEE\mathcal{L}^h_{P}(\nabla^2_h y)  + \partial_{\dot{F}}R(\nabla_h y,\partial_t \nabla_h y)  \Big) =   h^{3}  fe_3  & \text{in } [0,\infty) \times \Omega \\
y(0,\cdot) = y_0 & \text{in } \Omega \\
y(t,\cdot) \in \mathfrak{S}_h &\text{for } t\in [0,\infty)
\end{cases}
\end{align}
for some $y_0 \in \mathfrak{S}_h$, where $\partial_FW(\nabla_h y)$  $+h^{\beta}\mathcal{L}^h_{P}(\nabla^2_h y)$ denotes the    \emph{first Piola-Kirchhoff stress tensor} and $\partial_{\dot{F}}R(\nabla_h y,\partial_t \nabla_h y)$ the \emph{viscous stress} with $R$ as introduced in \eqref{intro:R}.

\BBB 

We remark that the   scaling of the forces corresponds to the so-called \emph{von K\'arm\'an regime}. The choice $0<\beta<4$ ensures that  the second-gradient term in the energy vanishes in the effective 2D limiting model as $h\to 0$. \EEE

 \smallskip

\textbf{Quadratic forms:}
To formulate the effective 2D problem, we need to consider various quadratic forms. First, we define \EEE $Q_W^3:\R^{3 \times 3} \to \R$  by $Q_W^3(F) = \partial^2_{F^2} W(\Id)[F,F]$. One can show that it depends only on the symmetric part $\frac{1}{2}(F^\top + F)$ and that it is positive definite on $\R^{3 \times 3}_{\rm sym} = \lbrace A  \in \R^{3 \times 3}: A=A^\top \rbrace$. \EEE We also introduce  $Q_W^2: \R^{2 \times 2} \to \R$ by
\begin{align}\label{eq:Q2}
Q_W^2(G) = \min_{a \in\R^3} Q_W^3(G^* + a \otimes e_3 + e_3 \otimes a )
\end{align}
for $G \in \R^{2 \times 2}$, where the entries of  $G^* \in \R^{3 \times 3}$ are given by $G^*_{ij} = G_{ij}$ for $i,j\in \lbrace 1,2\rbrace$ and zero otherwise. Note that \eqref{eq:Q2} corresponds to a minimization over stretches in the $e_3$ direction. \BBB In  \cite{Friedrich-Kruzik-19} it was  assumed \EEE that the minimum in \eqref{eq:Q2} is attained for $a=0$. Similarly, we define
\begin{align}\label{eq:Q22}
Q_D^3(F) =  \frac{1}{2}\partial^2_{F^2_1} D^2(\Id,\Id)[F,F],   \ \ \   Q_D^2(G) = \min_{a \in\R^3} Q_D^3(G^* + a \otimes e_3 + e_3 \otimes a ).
\end{align}
We again assume that the minimum is attained for $a=0$. The assumption that $a=0$ is a minimum in \eqref{eq:Q2}-\eqref{eq:Q22} corresponds to a model with  zero Poisson's ratio
in the $e_3$ direction. This assumption is not needed in the \PPP purely \EEE static analysis \cite{hierarchy, lecumberry}. \BBB However, it  is adopted  in \cite{Friedrich-Kruzik-19} to simplify the study of the evolutionary problem. \EEE  We also introduce  corresponding   symmetric   fourth order tensors \BBB $\C^2_W$ and $\C^2_D$  by
\begin{align}\label{eq: order4}
Q_W^2(G) = \C^2_W[G,G],  \ \ \ \ \ \ \ Q_D^2(G) = \C^2_D[G,G] \ \ \ \  \ \ \forall G \in \R^{2 \times 2}.
\end{align}
One can check that  $Q^2_W$ and $Q^2_D$  are positive semi-definite, and  positive definite on $\R_{\rm sym}^{2\times 2}$.
\EEE

\smallskip

 \textbf{Equations of viscoelasticity in 2D:}
\BBB We now  present  the effective 2D equations which are formulated in terms of in-plane and out-of-plane displacements fields $u$ and $v$.   Following the discussion in \cite{hierarchy}, these displacement fields can be related to the deformation $y$ in the three-dimensional setting by 
\begin{align*} 
u(x')  := \frac{1}{h^2} \int_I \Big( \begin{pmatrix}
 y_1 \\  y_2  \end{pmatrix} (x',x_3) - \begin{pmatrix}
x_1\\ x_2 \end{pmatrix} \Big) \, dx_3, \ \ \ \ \   v(x') := \frac{1}{h} \int_I  y_3  (x',x_3)\, dx_3,
\end{align*}
where again $I=(-\frac{1}{2},\frac{1}{2})$.  \EEE  Let us consider  \BBB the set of admissible displacement fields \EEE
\begin{align}\label{eq: BClinear}
{\mathscr{S}} = \lbrace (u,v) \in W_0^{1,2}(S;\R^2) \times W_0^{2,2}(S)\rbrace.
\end{align}
\BBB (Compare with \eqref{eq: nonlinear boundary conditions}.)  From now on,  we are going to work exclusively on the domain $S\subset \R^2$ and therefore $\nabla$ will denote  the gradient with respect to $x_1$ and $x_2$, i.e., we will drop the apostrophe from the notation.  
\PPP 

Given $(u_0,v_0) \in \mathscr{S}$, \EEE we consider the equations 
\begin{align}\label{eq: equation-simp}
\begin{cases}  {\rm div}\Big(\C^2_W\big( e(u)  + \frac{1}{2} \nabla v \otimes \nabla v  \big) +  \C^2_D \big( e(\partial_t u) +   \nabla \partial_t v   \odot\EEE \nabla v \big) \Big)   = 0, &\vspace{0.1cm} \\ 
-{\rm div}\Big(\Big(\C^2_W\big( e(u)  + \frac{1}{2} \nabla v \otimes \nabla v \big)  +  \C^2_D \big( e(\partial_t u) +   \nabla \partial_t v  \odot \nabla v \big) \Big) \nabla v \Big)  &\vspace{0.1cm}\\
 \quad\quad\quad\quad\quad    + \tfrac{1}{12}  {\rm div} \, {\rm div}\Big( \C^2_W \nabla^2 v + \C^2_D \nabla^2 \partial_t v \Big)   =   f  & \hspace{-1.8cm} \text{in } [0,\infty) \times S \\
u(0,\cdot) = u_0, \  v(0,\cdot) = v_0 & \hspace{-1.8cm} \text{in } S \\
(u(t,\cdot), v(t,\cdot)) \in {\mathscr{S}}_0 & \hspace{-1.9cm} \text{ for } t\in [0,\infty)
\end{cases}
\end{align}
where $\C^2_W$ and  $\C^2_D$ are defined in \eqref{eq: order4},  and $\odot$ denotes the symmetrized tensor product. \EEE Note that the frame indifference of the energy and the dissipation (see \eqref{assumptions-W}(ii) and \eqref{eq: assumptions-D}(v), respectively) imply that the contributions only depend on the symmetric part of the strain $e(u) := \frac{1}{2}(  \nabla u  +(\nabla u)^\top)$ and the strain rate  $e(\partial_t u) := \frac{1}{2}( \partial_t \nabla u + \partial_t (\nabla u)^\top)$. Here, \BBB ${\rm div}$ \EEE denotes the  distributional   divergence in dimension two. 

We also say that $(u,v) \in W^{1,2}([0,\infty);{\mathscr{S}})$ is a \emph{weak solution} of \eqref{eq: equation-simp}  if $u(0,\cdot) = u_0$, $v(0,\cdot) = v_0$ and for a.e.\ $t \ge 0$ we have 
 \begin{subequations}\label{eq: weak equation}
\begin{align}
& \int_S \Big(\C^2_W\big( e(u)  + \tfrac{1}{2} \nabla v \otimes \nabla v  \big)  +  \C^2_D \big( e(\partial_t u ) +  \nabla \partial_t v   \odot\EEE \nabla v \big) \Big) : \nabla \varphi_u    = 0,\label{eq: weak equation1} \\
& \int_S \Big(\C^2_W\big( e(u)  + \tfrac{1}{2} \nabla v \otimes \nabla v  \big)    \Big) : \big(\nabla v  \odot  \nabla \varphi_v  \big) \notag \\
&   \quad\quad\quad\quad   +  \int_S \Big( \C^2_D \big( e(\partial_t u ) +  \nabla \partial_t v  \odot\nabla v \big) \Big) : \big(\nabla v  \odot  \nabla \varphi_v  \big) \notag \\
& \quad\quad\quad\quad    + \frac{1}{12}  \int_S \Big(\C^2_W \nabla^2 v + \C^2_D \nabla^2 \partial_t v \Big) : \nabla^2 \varphi_v = \int_S f\varphi_v, \label{eq: weak equation2} 
\end{align}
\end{subequations} 
for all $\varphi_u \in W^{1,2}_0(S;\R^2)$ and  $\varphi_v \in W^{2,2}_0(S)$. Note that \eqref{eq: weak equation1} corresponds to two and \eqref{eq: weak equation2} corresponds to one equation, respectively. It is proved in \cite[Thm.~2.2 and Thm.~2.3]{Friedrich-Kruzik-19} that solutions to a semidiscretized-in-time system  \eqref{nonlinear equation} converge to   \BBB weak  solutions (in the sense of \eqref{eq: weak equation}) to the initial-boundary value problem  \eqref{eq: equation-simp}.    \EEE
 
 The following   \emph{von K\'arm\'an energy functional} $\phi:\mathscr{S}\to\R$ and the  \emph{global dissipation distance} $\mathcal{D}:\mathscr{S}\times \mathscr{S}\to\R$ due to viscosity will play an important role in our analysis: we define 
\begin{align}\label{eq: phi0}
{\phi}(u,v) := \int_S \frac{1}{2}Q_W^2\Big( e(u) + \frac{1}{2} \nabla v \otimes \nabla v \Big) + \frac{1}{24}Q_W^2(\nabla^2 v) - \int_S f v
\end{align}
for $(u,v) \in {\mathscr{S}}$ and 
 \begin{align}\label{eq: D,D0-2} 
  {\mathcal{D}}( (u_0,v_0),(u_1,v_1)) & := \Big(\int_S Q^2_D\Big( e( u_1) - e(u_0) + \frac{1}{2} \nabla v_1 \otimes \nabla v_1 - \frac{1}{2} \nabla v_0 \otimes \nabla v_0 \Big) \notag \\
  & \ \ \ \ \  +  \frac{1}{12}   Q_D^2\big(\nabla^2 v_1 - \nabla^2 v_0 \big)  \Big)^{1/2}
\end{align}
for  $(u_0, v_0), (u_1,v_1) \in {\mathscr{S}}$.

\smallskip

\BBB In the next sections, we provide mathematical tools which will be used to show that fully discretized  solutions to  \eqref{eq: equation-simp} (i.e., discretized in time and space) \EEE converge to 
 weak solutions (in the sense of  \eqref{eq: weak equation}) to \eqref{eq: equation-simp}. 

\section{\BBB An abstract convergence result}\label{sec3}

  In this section we first recall the relevant definitions \BBB for metric gradient flows. Then, based on \cite{Ortner}, we prove an abstract convergence result of time-space discretizations  to curves of maximal slope. \EEE   

\subsection{Definitions: \BBB Curves of maximal slope and time-discrete solutions\EEE}\label{sec: defs}

We consider a   complete metric space $(\mathscr{S},\mathcal{D})$. We say a curve $u: (a,b) \to \mathscr{S}$ is \emph{absolutely continuous} with respect to $\mathcal{D}$ if there exists $m \in L^1(a,b)$ such that
$$
\mathcal{D}(u(s),u(t)) \le \int_s^t m(r) \, dr \ \ \   \text{for all} \ a \le s \le t \le b.
$$
The smallest function $m$ with this property, denoted by $|u'|_{\mathcal{D}}$, is called the \emph{metric derivative} of  $u$  and satisfies  for a.e.\ $t \in (a,b)$   (see \cite[Theorem 1.1.2]{AGS} for the existence proof)
$$|u'|_{\mathcal{D}}(t) := \lim_{s \to t} \frac{\mathcal{D}(u(s),u(t))}{|s-t|}.$$
We now define the notion of a \emph{curve of maximal slope}. We only give the basic definition here and refer to \cite[Section 1.2, 1.3]{AGS} for motivations and more details.  By  $h^+:=\max(h,0)$ we denote the positive part of a function  $h$.

\begin{definition}[Upper gradients, slopes, curves of maximal slope]\label{main def2} 
 We consider a   complete metric space $(\mathscr{S},\mathcal{D})$ with a functional $\phi: \mathscr{S} \to (-\infty,+\infty]$.

(i) A function $g: \mathscr{S} \to [0,\infty]$ is called a strong upper gradient for $\phi$ if for every absolutely continuous curve $v: (a,b) \to \mathscr{S}$ the function $g \circ v$ is Borel and 
$$|\phi(v(t)) - \phi(v(s))| \le \int_s^t g(v(r)) |v'|_{\mathcal{D}}(r)\,dr \  \ \  \text{for all} \ a< s \le t < b.$$

(ii) For each $u \in \mathscr{S}$ the local slope of $\phi$ at $u$ is defined by 
$$|\partial \phi|_{\mathcal{D}}(u): = \limsup_{w \to u} \frac{(\phi(u) - \phi(w))^+}{\mathcal{D}(u,w)}.$$

(iii) An absolutely continuous curve $u: (a,b) \to \mathscr{S}$ is called a curve of maximal slope for $\phi$ with respect to the strong upper gradient $g$ if for a.e.\ $t \in (a,b)$
$$\frac{\rm d}{ {\rm d} t} \phi(u(t)) \le - \frac{1}{2}|u'|^2_{\mathcal{D}}(t) - \frac{1}{2}g^2(u(t)).$$
\end{definition}

 We introduce time-discrete solutions for a \BBB functional $\phi: \mathscr{S} \to (-\infty,+\infty]$ and \EEE the metric $\mathcal{D}$ by solving suitable time-incremental minimization problems: consider a fixed time step $\tau >0$ and suppose that an initial datum $Y^0_{\tau}$ is given. Whenever $Y_{\tau}^0, \ldots, Y^{n-1}_{\tau}$ are known, $Y^n_{\tau}$ is defined as (if existent)
\begin{align}\label{eq: ds-new1}
Y_{\tau}^n = {\rm argmin}_{v \in \mathscr{S}} \  \Phi(\tau,Y^{n-1}_{\tau}; v), \ \ \ \Phi(\tau,u; v):=  \frac{1}{2\tau} \mathcal{D}(v,u)^2 + \phi(v). 
\end{align}
We suppose that for a choice of $\tau$ a sequence $(Y_{\tau}^n)_{n \in \N}$ solving  \eqref{eq: ds-new1} \EEE exists. Then we define the  piecewise constant interpolation by
\begin{align}\label{eq: ds-new2}
 \tilde{Y}_{\tau}(0) = Y^0_{\tau}, \ \ \ \tilde{Y}_{\tau}(t) = Y^n_{\tau}  \ \text{for} \ t \in ( (n-1)\tau,n\tau], \ n\ge 1.  
\end{align}
We call  $\tilde{Y}_{\tau}$  a \emph{time-discrete solution}. Note that the existence of such  solutions is usually guaranteed by the direct method of the calculus of variations under suitable compactness, coercivity, and lower semicontinuity assumptions.

\subsection{Curves of maximal slope as limits of time-space discretizations}

In this subsection we formulate a result about the approximation of curves of maximal slope. \BBB It \EEE is based on a result in \cite{Ortner} recalled in Subsection  \ref{sec: Ortner} below. We first state our assumptions. We again consider a complete metric space $(\mathscr{S},\mathcal{D})$ and a functional $\phi: \mathscr{S} \to [0,\infty]$. Although $\mathcal{D}$ naturally induces a topology \BBB on \EEE $\mathscr{S}$, it is often convenient to consider a weaker Hausdorff topology $\sigma$ on $\mathscr{S}$ to have more flexibility in the derivation of compactness properties (see \cite[Remark 2.0.5]{AGS}). We assume that \BBB for each $n \in \N$ \EEE there exists a  $\sigma$-sequentially   compact set   $K_N \subset \mathscr{S}$   such that  
\begin{align}\label{basic assumptions2-new}
\lbrace z  \in \mathscr{S}: \ \phi(z) \le N \rbrace \subset K_N.
\end{align}
Moreover, we suppose that the topology $\sigma$ satisfies  
\begin{align}\label{compatibility-new}
(i) & \ \ z_k \stackrel{\sigma}{\to} z, \ \  w_k \stackrel{\sigma}{\to} w  \ \ \  \Rightarrow \ \ \ \liminf_{k \to \infty} \mathcal{D}(z_k,w_k) \ge  \mathcal{D}(z,w), \notag \\
(ii) & \ \ z_k \stackrel{\sigma}{\to} z    \ \ \  \Rightarrow \ \ \ \liminf_{k \to \infty} \phi(z_k) \ge  \phi(z).
\end{align}
We further assume the existence of \emph{mutual recovery sequences}: for each sequence $z_k \stackrel{\sigma}{\to} z$  and $w \in \mathscr{S}$ there exists a sequence $(w_k)_k \subset \mathscr{S}$ such that
\begin{align}\label{eq: mutual recovery}
\limsup_{k\to \infty}\mathcal{D}(z_k,w_k) \le \mathcal{D}(z,w), \ \ \ \ \ \  \phi(z) - \phi(w) \le \liminf_{k \to \infty} \big(  \phi(z_k) - \phi(w_k)  \big).
\end{align}  
This condition is reminiscent of \cite[(2.1.37)]{MR}. \BBB We also point out that \EEE this assumption is weaker than the one considered in \cite[(2.26)-(2.27)]{Ortner}.

   We consider a sequence of  subspaces  $\mathscr{S}_k \subset \mathscr{S}$, $k \in \N$, such that each $\mathscr{S}_k$ is closed with respect to the topology $\sigma$. By $\rho$ we denote a stronger topology on $\mathscr{S}$ with the property that $\phi$ and $\mathcal{D}$ are continuous with respect to $\rho$.   We suppose that $\bigcup_k\mathscr{S}_k$ is $\rho$-dense in $\mathscr{S}$, i.e., for each $z\in \mathscr{S}$ we find a sequence $\BBB (z_k)_k \EEE \in \mathscr{S}$ such that 
\begin{align}\label{eq: strong convergence}
z_k \stackrel{\rho}{\to} z.
\end{align}
In our applications, $\mathscr{S}_k$ will represent finite element subspaces.

Finally, we require a property about \emph{geodesical convexity}: let $M>0$ and let   $\Theta^1_M, \Theta^2_M:[0,+\infty) \to [0,+\infty)$ be continuous, increasing functions which satisfy $\lim_{t \to 0} \Theta^1_M(t)/t = 1$ and $\lim_{t \to 0} \Theta^2_M(t)/t = 0$. We suppose that for all $z_0,z_1 \in\mathscr{S}$ with $\phi(z_0) \le M$ there exists a curve $(\gamma_s)_{s \in[0,1]} \subset \mathscr{S}$ with $\gamma_0 = z_0$ and $\gamma_1 = z_1$ such that 
\begin{align}\label{eq: goedesi convexi}
(i)& \ \    \mathcal{D}\big(z_0, \gamma_s\big)  \le  s \, \Theta^1_M\big(\mathcal{D}(z_0,z_1)\big), \notag \\
(ii) & \ \ \phi(\gamma_s)  \le (1-s) \phi(z_0) + s\phi(z_1) +s \, \Theta^2_M\big(\mathcal{D}(z_0,z_1)\big). 
\end{align}
Moreover, we assume that, if $z_0,z_1$ in $\mathscr{S}_k$, then $(\gamma_s)_{s \in[0,1]} \subset \mathscr{S}_k$, as well. In our applications, these curves will simply be convex combinations and, in this context, we will exploit that the finite element spaces $\mathscr{S}_k$ are obviously  convex sets.

\begin{remark}[Convexity assumption on $\Phi$]
We mention that the condition presented here \cm is tailor-made for our applications to the viscoelastic plate model since in this case we can find curves satisfying \eqref{eq: goedesi convexi} for specific $\Theta_M^1$ and $\Theta_M^2$, see Lemma \ref{th: convexity2} below.  We point out that the condition \EEE is slightly more general than the one used in   \cite[Assumption 2.4.5]{AGS} or \cite[Assumption 9]{Ortner}: fix $\lambda \in \R$ and let $\lambda^- = -\lambda$ for $\lambda \le 0$ and $1/\lambda^- = +\infty$ else.  We suppose that there exists a curve $(\gamma_s)_{s \in[0,1]} \subset \mathscr{S}$ with $\gamma_0 = z_0$ and $\gamma_1 = z_1$ such that for all $\tau \in (0, 1/\lambda^-)$ and all $s \in [0,1]$ there holds
\begin{align}\label{eq: goedesi convexi2}
\Phi(\tau,z_0; \gamma_s) \le (1-s)\Phi(\tau,z_0;\gamma_0) + s  \, \Phi(\tau,z_0;\gamma_1) -\frac{1}{2}\Big( \frac{1}{\tau} + \lambda \Big) s(1-s)  \mathcal{D}(z_0,z_1)^2.  
\end{align}
\BBB Note here that the curve $(\gamma_s)_{s}$ is chosen independently of $\tau$. \EEE A prototypical case is the case of $\lambda$-geodesically convex functionals $\phi$, see \cite[Definition 2.4.3]{AGS}.  We briefly check that \eqref{eq: goedesi convexi2} implies \eqref{eq: goedesi convexi}.

In view of \eqref{eq: ds-new1}, multiplying \eqref{eq: goedesi convexi2} with $2\tau$ and passing to the limit $\tau \to 0$ we obtain
$$\mathcal{D}(z_0, \gamma_s)^2 \le (1-s) \mathcal{D}(z_0,\gamma_0)^2     +s \mathcal{D}(z_0,\gamma_1)^2 - s(1-s)  \mathcal{D}(z_0,z_1)^2 = s^2 \mathcal{D}(z_0,z_1)^2,$$
i.e., \eqref{eq: goedesi convexi}(i) holds for $\Theta_M^1(t) = t$. On the other hand, for $\tau \nearrow 1/\lambda^-$, we get
\begin{align*}
\phi(\gamma_s)& \le \Phi(1/\lambda^-,z_0; \gamma_s) \le (1-s)\Phi(1/\lambda^-,z_0;\gamma_0) + s \Phi(1/\lambda^-,z_0;\gamma_1) \\
&\le (1-s)\phi(z_0) + s\phi(z_1) + s \frac{\lambda^-}{2} \mathcal{D}(z_0,z_1)^2,
\end{align*} 
 i.e., \eqref{eq: goedesi convexi}(ii) holds for $\Theta_M^2(t) = \frac{\lambda^-}{2} t^2$. \cm In this sense, \eqref{eq: goedesi convexi2} can be understood as a special case of \eqref{eq: goedesi convexi} with  $\Theta_M^1$ being the identity and $\Theta_M^2$ being quadratic.  \EEE
\end{remark}

We now state or main approximation result. Recall the definition of time-discrete solutions  $\tilde{Y}_{\tau}$ in 
\eqref{eq: ds-new1}-\eqref{eq: ds-new2}. We say that  $\tilde{Y}_{\tau}$ is a time-discrete solution in $\mathscr{S}_k$ if $\tilde{Y}_{\tau}(0) \in \mathscr{S}_k$ and  the minimization problem in \eqref{eq: ds-new1} is restricted to $\mathscr{S}_k$.

     \begin{theorem}\label{th:abstract convergence-numerics}
     Let  $(\mathscr{S},\mathcal{D})$ be a complete metric space and let $\phi: \mathscr{S} \to [0,\infty]$. Consider topologies $\tau$ and $\rho$ on $\mathscr{S}$ such that $\mathcal{D}$ and $\phi$ are continuous with respect to $\rho$. Consider $\sigma$-closed subspaces $\mathscr{S}_k\subset \mathscr{S}$ and suppose that   \eqref{basic assumptions2-new}-\eqref{eq: goedesi convexi} hold.   Consider a  null sequence $(\tau_k)_k$.  Let $\bar{z}_0 \in \mathscr{S}$. 
     
     Then there exist initial values  $(Y^0_{k,\tau_k})_k$ satisfying $Y^0_{k,\tau_k} \in \mathscr{S}_k$  and  
\begin{align}\label{eq: abstract assumptions1-new}
Y^0_{k,\tau_k} \stackrel{\sigma}{\to} \bar{z}_0 , \ \ \ \ \  \phi(Y^0_{k,\tau_k}) \to \phi(\bar{z}_0),
\end{align}
 sequences of time-discrete solutions $(\tilde{Y}_{k,\tau_k})_k$ in $\mathscr{S}_k$   starting from $(Y^0_{k,\tau_k})_k$, and a limiting curve  $z: [0,+\infty) \to \mathscr{S}$ such that up to a  subsequence (not relabeled) 
$$\tilde{Y}_{k,\tau_k}(t) \stackrel{\sigma}{\to} z(t), \ \ \ \ \ \phi(\tilde{Y}_{\tau_k}(t)) \to \phi(z(t)) \ \ \  \ \ \ \ \ \forall t \ge 0$$
as $k \to \infty$. The function $z$ is a curve of maximal slope for $\phi$ with respect to $|\partial \phi|_{\mathcal{D}}$. 

\end{theorem}

\subsection{Curves of maximal slope as limits of time-discrete solutions}\label{sec: Ortner}

In this \BBB subsection \EEE we recall a result about the limits of time-discrete solutions obtained by {\sc Ortner} \cite{Ortner} which is the main ingredient for the proof of Theorem \ref{th:abstract convergence-numerics}.  We consider a set $\mathscr{S}$ and  a sequence of  metrics $(\mathcal{D}_k)_k$ on $\mathscr{S}$ as well as a limiting metric $\mathcal{D}$. We again assume that all metric spaces are complete. Moreover, let $(\phi_k)_k$ be a sequence of functionals with $\phi_k: \mathscr{S} \to [0,\infty]$.  

As before, we consider a Hausdorff topology $\sigma$ on $\mathscr{S}$ which is possibly weaker than the one induced by $\mathcal{D}$. 
We suppose that the topology $\sigma$ satisfies  
\begin{align}\label{compatibility}
\begin{split}
z_k \stackrel{\sigma}{\to} z, &\ \  w_k \stackrel{\sigma}{\to} w  \ \ \  \Rightarrow \ \ \ \liminf_{k \to \infty} \mathcal{D}_k(z_k,w_k) \ge  \mathcal{D}(z,w).
\end{split}
\end{align}
Moreover, assume that  \BBB for all $n \in \N$ \EEE  there exists a  $\sigma$-sequentially   compact set   $K_N \subset \mathscr{S}$   such that for all $k \in \N$
\begin{align}\label{basic assumptions2}
\lbrace  z: \ z \in \mathscr{S}, \ \phi_k(z) \le N \rbrace \subset K_N.
\end{align}
Specifically, for a sequence $(z_k)_k$ with $\phi_k(z_k) \le N$, we find a subsequence (not relabeled) and $z \in \mathscr{S}$ such that $z_k \stackrel{\sigma}{\to} z$. We suppose lower semicontinuity of the energies and the slopes in the following sense: for all $z \in \mathscr{S}$ and \BBB sequences \EEE $(z_k)_k$, $z_k \in \mathscr{S}_k$, we have
 \begin{subequations}\label{eq: implication}
\begin{align}
z_k \stackrel{\sigma}{\to}  z \ \ \ \ \  &\Rightarrow \ \ \ \ \  \liminf_{k \to \infty} \phi_{k}(z_{k}) \ge \phi(z), \label{eq: implication-1} \\
z_k \stackrel{\sigma}{\to}  z, \ \ \   \sup\nolimits_k \phi_k(z_k) <+\infty  \  \ \  &\Rightarrow \ \ \ \ \  \liminf_{k \to \infty} |\partial \phi_{k}|_{\mathcal{D}_{k}} (z_{k}) \ge |\partial \phi|_{\mathcal{D}} (z). \label{eq: implication-2}
\end{align}
 \end{subequations}
We remark that the condition in \cite[(2.10)]{Ortner} is slightly stronger than \eqref{eq: implication-2} since there the condition is required for all sequences and not only on sublevel sets of $\phi_k$. The following results remain true under the weaker assumption \eqref{eq: implication-2}, cf., e.g., \cite[Corollary 2.4.12]{AGS}. Note that nonnegativity of $\phi_k$ and $\phi$ can be generalized to a suitable \emph{coerciveness} condition, see \cite[(2.1.2b)]{AGS} or \cite[(2.5)]{Ortner}, which we do not include here for the sake of simplicity. We formulate the main convergence result of time-discrete solutions to curves of maximal slope, proved in \cite[Section 2]{Ortner}.


\begin{theorem}\label{th:abstract convergence 2}
Suppose that   \eqref{compatibility}-\eqref{eq: implication} hold. Moreover, assume that    $|\partial \phi|_{\mathcal{D}}$ is a  strong upper gradient for $ \phi $.   Consider a  null sequence $(\tau_k)_k$. Let   $(Y^0_{k,\tau_k})_k$ with $Y^0_{k,\tau_k} \in \BBB \mathscr{S} \EEE $  and $\bar{z}_0 \in \mathscr{S}$ be initial data satisfying 
\begin{align}\label{eq: abstract assumptions1}
(i)& \ \ \sup\nolimits_k \mathcal{D} \big(Y^0_{k,\tau_k},\bar{z}_0\big) < + \infty, \notag \\ 
(ii)& \ \  Y^0_{k,\tau_k} \stackrel{\sigma}{\to} \bar{z}_0 , \ \ \ \ \  \phi_k(Y^0_{k,\tau_k}) \to \phi(\bar{z}_0).
\end{align}
Then for each sequence of discrete solutions $(\tilde{Y}_{k,\tau_k})_k$ for $\phi_k$ and $\mathcal{D}_k$  starting from $(Y^0_{k,\tau_k})_k$, \BBB see \eqref{eq: ds-new1}-\eqref{eq: ds-new2}, \EEE there exists a limiting function $z: [0,+\infty) \to \mathscr{S}$ such that up to a  subsequence (not relabeled) 
$$\tilde{Y}_{k,\tau_k}(t) \stackrel{\sigma}{\to} z(t), \ \ \ \ \ \phi_k(\tilde{Y}_{\tau_k}(t)) \to \phi(z(t)) \ \ \  \ \ \ \ \ \forall t \ge 0$$
as $k \to \infty$, and $z$ is a curve of maximal slope for $\phi$ with respect to $|\partial \phi|_{\mathcal{D}}$. 

\end{theorem}

For the proof we refer to \cite[Proposition 5, 6]{Ortner}. We comment that this convergence result might seem weak at first glance since in the family of approximations there exists only a subsequence \BBB converging \EEE to a solution. In practice, \BBB however, \EEE this often does not cause problems, see \cite[Remark 7]{Ortner} for a thorough comment.

\subsection{Proof of Theorem \ref{th:abstract convergence-numerics}}

This subsection is devoted to the proof of Theorem \ref{th:abstract convergence-numerics}. Consider the complete metric spaces $(\mathscr{S},\mathcal{D})$ and $(\mathscr{S}_k,\mathcal{D})$, the functional $\phi:\mathscr{S} \to [0,+\infty]$,  and recall assumptions    \eqref{basic assumptions2-new}-\eqref{eq: goedesi convexi}.  We start with a representation of the local slope defined in Definition \ref{main def2}. We also define $\phi_k:\mathscr{S} \to [0,+\infty]$ by $\phi_k(z) = \phi(z)$ if $z \in \mathscr{S}_k$  
 and $\phi_k(z) = +\infty$ else.

\begin{lemma}[Representation of the local slope]\label{lemma: slopes}
\BBB Let $M>0$. \EEE The local slope for the energy ${\phi}$ in the complete metric space $(\mathscr{S},\mathcal{D})$ admits the representation 
$$ |\partial {\phi}|_{{\mathcal{D}}}(z) =  \sup_{ w  \neq z, \, w \in {\mathscr{S}}} \   \frac{\big({\phi}(z) - {\phi}(w) - \Theta^2_M\big({\mathcal{D}}(z,w) \big) \big)^+}{\Theta^1_M\big({\mathcal{D}}(z,w)\big)} $$
 for all $z \in {\mathscr{S}}$ with $\phi(z) \le M$,   where $\Theta^1_M$ and $\Theta^2_M$ are the functions from  \eqref{eq: goedesi convexi}.  The local slope is a strong upper gradient for ${\phi}$. The same \BBB representation \EEE holds for $\phi_k$ in place of $\phi$. 
\end{lemma}

\begin{proof} 
We prove the result only for $\phi$. The argument for $\phi_k$  is exactly the same which we will explain briefly at the end of the proof. We follow  the lines of the proofs of Theorem 2.4.9 and Corollary 2.4.10 in \cite{AGS},  see also \cite[Lemma 4.9]{Friedrich-Kruzik-19}. \EEE Let  $M>0$ and $z \in \mathscr{S}$ with $\phi(z) \le M$. Recall that $\lim_{t \to 0}\Theta^1_M(t)/t =1$ and $\lim_{t \to 0}\Theta^2_M(t)/t =0$.  We also recall the definition of the local slope in Definition \ref{main def2} and obtain 
\begin{align*}
|\partial {\phi}|_{{\mathcal{D}}}(z) & = \limsup_{w  \to z} \frac{({\phi}(z) - {\phi}(w))^+}{{\mathcal{D}}(z,w)}      
= \limsup_{w \to z} \frac{\big({\phi}(z) - {\phi}(w) - \Theta^2_M\big({\mathcal{D}}(z,w)\big)\big)^+}{\Theta_M^1\big({\mathcal{D}}(z,w)\big)}   \\
&\le \sup_{w  \neq z, w \in \mathscr{S}} \  \frac{\big({\phi}(z) - {\phi}(w) - \Theta^2_M\big({\mathcal{D}}(z,w)\big)\big)^+}{\Theta_M^1\big({\mathcal{D}}(z,w)\big)}.
\end{align*}
In the second equality we used  that $w \to z$ means ${\mathcal{D}}(z,w) \to 0$, and the fact that   $\lim_{t \to 0}\Theta_M^1(t)/t =1$ and $\lim_{t \to 0}\Theta^2_M(t)/t =0$.

 To see the other inequality,  we fix $w \neq z$. It   is not restrictive to suppose that 
\begin{align*}
\phi(z) - \phi(w) - \Theta^2_M\big({\mathcal{D}}(z,w)\big)>0.
\end{align*}
Let $(\gamma_s)_{s \in [0,1]}$ be the curve given in \eqref{eq: goedesi convexi} with $\gamma_0 = z$ and $\gamma_1 = w$. By \eqref{eq: goedesi convexi} we obtain 
 \begin{align*}
\frac{\phi(z) - \phi(\gamma_s)}{\mathcal{D}(z,\gamma_s )}  &\ge \frac{ s{\phi}(z) - s{\phi}(w) - s\Theta^2_M\big({\mathcal{D}}(z,w)\big) }{s\Theta_M^1\big({\mathcal{D}}(z,w)\big)}. 
\end{align*}
Since  $\mathcal{D}(\gamma_s, z)  \to 0$ as $s \to 0$, \BBB see \eqref{eq: goedesi convexi}(i), \EEE we conclude 
$$|\partial {\phi}|_{{\mathcal{D}}}(z)   \ge \frac{ {\phi}(z) - {\phi}(w) -\Theta^2_M\big({\mathcal{D}}(z,w)\big) }{\Theta^1_M\big({\mathcal{D}}(z,w)\big)}. $$
The claim now follows by taking the supremum with respect to $w \in \mathscr{S}$.

With this representation of the local slope at hand, one can also show that   $|\partial {\phi}|_{{\mathcal{D}}}$ is a strong upper gradient. We refer the reader to   \cite[Corollary 2.4.10]{AGS} and \cite[Lemma 4.9]{Friedrich-Kruzik-19} for details. 

The same argument works for $\phi_k$ in place of $\phi$. The only important point to notice is that the curve $(\gamma_s)_{s \in [0,1]}$ lies in $\mathscr{S}_k$ if $z,w \in \mathscr{S}_k$, see the line below  \eqref{eq: goedesi convexi}, i.e., $\phi_k(\gamma_s) = \phi(\gamma_s)$ for $s \in [0,1]$. 
\end{proof}

We  are now ready for the proof of Theorem~\ref{th:abstract convergence-numerics}.

\begin{proof}[Proof of Theorem \ref{th:abstract convergence-numerics}]

  Consider a  null sequence $(\tau_k)_k$ sequence and $\bar{z}_0 \in \mathscr{S}$. By \eqref{eq: strong convergence} and the fact that $\phi$ is continuous with respect to $\rho$ \cm (to recall its definition, see the paragraph preceding \eqref{eq: strong convergence}) \EEE we find a sequence  $(Y^0_{k,\tau_k})_k$ satisfying $Y^0_{k,\tau_k} \in \mathscr{S}_k$,    $Y^0_{k,\tau_k} \stackrel{\BBB \rho \EEE}{\to}  \bar{z}_0$, and $\phi(Y^0_{k,\tau_k}) \to \phi(\bar{z}_0)$. This yields   \eqref{eq: abstract assumptions1-new} and also  \eqref{eq: abstract assumptions1}(ii). Since also $\mathcal{D}$ is continuous with respect to $\rho$, \eqref{eq: abstract assumptions1}(i) holds as well.

Recall the definition $\phi_k:\mathscr{S} \to [0,+\infty]$ by $\phi_k(z) = \phi(z)$ if $z \in \mathscr{S}_k$  
 and $\phi_k(z) = +\infty$ else. We define time-discrete solutions $(\tilde{Y}_{k,\tau_k})_k$ in the sense of \eqref{eq: ds-new1}-\eqref{eq: ds-new2} with respect to $\phi_k$ starting from $(Y^0_{k,\tau_k})_k$.  Their existence follows from the direct method of the calculus of variations, by using  \eqref{basic assumptions2-new}, \eqref{compatibility-new}, and the fact that $\mathscr{S}_k$ is closed with respect to $\sigma$. Clearly, these correspond to time-discrete solutions for $\phi$ in $\mathscr{S}_k$. 

It remains to check that the time-discrete solutions converge to a limiting curve which is a curve of maximal slope for $\phi$ with respect to $|\partial \phi|_{\mathcal{D}}$. Our goal is to apply Theorem \ref{th:abstract convergence 2}. Since \eqref{eq: abstract assumptions1} has already been verified and $|\partial {\phi}|_{{\mathcal{D}}}$ is a strong upper gradient by Lemma \ref{lemma: slopes}, it remains to confirm \eqref{compatibility}-\eqref{eq: implication}.  Set  $\mathcal{D}_k = \mathcal{D}$  for all $k \in \N$. First, \eqref{compatibility} and \eqref{basic assumptions2} follow from the fact that $\phi \le \phi_k$, \eqref{basic assumptions2-new}, and \eqref{compatibility-new}(i). In a similar fashion, \eqref{eq: implication-1}  follows from \eqref{compatibility-new}(ii). We now show \eqref{eq: implication-2}.
 
Consider a sequence $z_k \in \mathscr{S}_k$ with $\sup_k\phi_k(z_k) \le M < +\infty$ and   $z_k\stackrel{\sigma}{\to} z$. By \eqref{eq: implication-1} we find also $\phi(z) \le M$. Let $\eps > 0 $. By applying Lemma \ref{lemma: slopes} we choose $w \in \mathscr{S}$  such that
\begin{align}\label{eq: starting point} 
|\partial {\phi}|_{{\mathcal{D}}}(z) \le   \frac{\big({\phi}(z) - {\phi}(w) - \Theta^2_M\big({\mathcal{D}}(z,w) \big) \big)^+}{\Theta^1_M\big({\mathcal{D}}(z,w)\big)} +\eps.
\end{align}
Let $(w_k)_k \subset \mathscr{S}$ be a mutual recovery sequence as given by \eqref{eq: mutual recovery}. By \eqref{eq: strong convergence} and the fact that $\mathcal{D}$ and $\phi$ are continuous with respect to the topology $\rho$, we can suppose that $w_k \in \mathscr{S}_k$  and  convergence \eqref{eq: mutual recovery} still holds.  By \eqref{eq: starting point}, $\phi(z_k)= \phi_k(z_k)$, $\phi(w_k)= \phi_k(w_k)$,  and the fact that $\Theta^i_M$ is continuous, increasing for $i=1,2$,  we then  obtain 
\begin{align*}
 |\partial {\phi}|_{{\mathcal{D}}}(z) -\eps \le \BBB \liminf_{k\to \infty}\EEE\frac{\big({\phi}_k(z_k) - {\phi}_k(w_k) - \Theta^2_M\big({\mathcal{D}}(z_k,w_k) \big) \big)^+}{\Theta^1_M\big({\mathcal{D}}(z_k,w_k)\big)}.
\end{align*}
By Lemma \ref{lemma: slopes} (for $\phi_k$)  we then get
\begin{align*}
 |\partial {\phi}|_{{\mathcal{D}}}(z) -\eps &\le \liminf_{k\to \infty} \sup_{w \neq z_k, w \in \mathscr{S}}\frac{\big({\phi}_k(z_k) - {\phi}_k(w) - \Theta^2_M\big({\mathcal{D}}(z_k,w) \big) \big)^+}{\Theta^1_M\big({\mathcal{D}}(z_k,w)\big)} \\ 
 & \BBB = \EEE \liminf_{k \to \infty}   |\partial {\phi}_k|_{{\mathcal{D}}}(z_k).
\end{align*}
As $\eps$ was arbitrary, we get    \eqref{eq: implication-2}.

The statement  now follows from the abstract convergence result formulated in  Theorem \ref{th:abstract convergence 2}. 
\end{proof}

\section{Finite element approximation of weak solutions to von K\'{a}rm\'{a}n viscoelastic plates}\label{sec:FEM}

In this section we apply  Theorem \ref{th:abstract convergence-numerics} to our example of von K\'{a}rm\'{a}n viscoelastic plates. Let $\mathscr{S} = W^{1,2}_0(S;\R^2) \times W^{2,2}_0(S)$,  see \eqref{eq: BClinear}. We denote the strong convergence in  $W^{1,2}(S;\R^2) \times W^{2,2}(S)$ by $\stackrel{\rho}{\to}$. Moreover, we introduce a weak  topology $\sigma$ on $\mathscr{S}$: we say that $(u_k,v_k)   \stackrel{\sigma}{\to} (u,v)$ if $u_k \rightharpoonup u$ weakly in $W^{1,2}(S;\R^2)$ and $v_k \rightharpoonup  v$ weakly in $W^{2,2}(S)$.  We let $\mathscr{S}_k \subset \mathscr{S}$ be finite dimensional subspaces of finite elements such that $\mathscr{S}_k$ is closed with respect to $\sigma$ and $\bigcup_k\mathscr{S}_k$ is $\rho$-dense in $\mathscr{S}$ in the sense of \eqref{eq: strong convergence}. \BBB For an example of such spaces we refer to Section \ref{numexp} below. \EEE

 We let $\phi$ and $\mathcal{D}$ as defined in \eqref{eq: phi0} and \eqref{eq: D,D0-2}, respectively. For simplicity, we set $f \equiv 0$ since the adaptions for the general case are minor and standard.

We recall that  $\tilde{Y}_{k,\tau}$ is called  a time-discrete solution in $\mathscr{S}_k$ if $\tilde{Y}_{k,\tau}(0) \in \mathscr{S}_k$ and  the minimization problem in \eqref{eq: ds-new1} is restricted to $\mathscr{S}_k$. Our main result is the following.

\begin{theorem}[Finite element approximation of weak solutions]\label{maintheorem}
 Consider a  null sequence $(\tau_k)_k$ and let $(u_0,v_0) \in {\mathscr{S}}$. Then  there exist initial values  $(U^0_{k,\tau_k})_k$,  $(V^0_{k,\tau_k})_k$ satisfying $(U^0_{k,\tau_k},U^0_{k,\tau_k})_k \in \mathscr{S}_k$ and 
  \begin{align*}
(U^0_{k,\tau_k},V^0_{k,\tau_k}) \stackrel{\rho}{\to} (u_0,v_0) , \ \ \ \ \  \phi(U^0_{k,\tau_k},U^0_{k,\tau_k}) \to \phi(u_0,v_0),
\end{align*}
  sequences of time-discrete solutions $(\tilde{U}_{k,\tau_k}, \tilde{V}_{k,\tau_k})_k$ in $\mathscr{S}_k$   starting from the initial values $(U^0_{k,\tau_k},U^0_{k,\tau_k})_k$, and a  weak solution $(u,v) :[0,\infty) \to {\mathscr{S}}$  to the partial differential equations \eqref{eq: equation-simp}   in the sense of \eqref{eq: weak equation}  such that 
  up to a  subsequence (not relabeled) 
$$\big(\tilde{U}_{k,\tau_k}(t), \tilde{V}_{k,\tau_k}(t) \big) \stackrel{\rho}{\to} (u(t),v(t)), \ \ \ \ \ \phi\big(\tilde{U}_{k,\tau_k}(t), \tilde{V}_{k,\tau_k}(t) \big) \to \phi(u(t),v(t)) \ \ \  \ \ \ \ \ \forall t \ge 0$$
as $k \to \infty$.  

\end{theorem}

Note that this theorem provides us with the strong convergence of time-discrete finite-element \BBB approximations \EEE to a solution to the original problem. 

The result relies on our abstract approximation result \BBB stated in \EEE Theorem \ref{th:abstract convergence-numerics}. In order to apply  Theorem \ref{th:abstract convergence-numerics}, we need to check the assumptions \eqref{basic assumptions2-new}-\eqref{eq: goedesi convexi}. To this end, we recall some of the results obtained in \cite{Friedrich-Kruzik-19}.

\begin{lemma}[Properties of $({\mathscr{S}},  {\mathcal{D}})$ and ${\phi}$]\label{th: metric space-lin}
We have:

\begin{itemize}
\item[(i)] $({\mathscr{S}},  {\mathcal{D}})$ is a complete metric space.
\item[(ii)] Compactness: If $(u_k,v_k)_k \subset {\mathscr{S}}$ is a sequence with $\sup_k {\phi}(u_k,v_k)<+\infty$, then $(u_k,v_k)_k$ is bounded in $W^{1,2}(S;\R^2) \times W^{2,2}(S)$. 
\item[(iii)] Topologies: The topology induced by ${\mathcal{D}}$ is equivalent to the topology $\rho$.    
\item[(iv)] Continuity: ${\mathcal{D}}( (u_k,v_k), (u,v)) \to 0$ \ \   $\Rightarrow$ \ \  $\lim_{k \to \infty} {\phi}(u_k,v_k) =  {\phi}(u,v)$.
\end{itemize}
\end{lemma}

\begin{proof}
See \cite[Lemma 4.6]{Friedrich-Kruzik-19}.
\end{proof}

\begin{theorem}[Curves of maximal slope and weak solutions]\label{maintheorem2}
For all $(u_0,v_0) \in {\mathscr{S}}$,  each  curve of maximal slope $(u,v) :[0,\infty) \to {\mathscr{S}}$  for ${\phi}$   with respect to  $|\partial {\phi}|_{{\mathcal{D}}}$ with $(u,v)(0)=(u_0,v_0)$  is a weak solution  to the partial differential equations \eqref{eq: equation-simp} in the sense of \eqref{eq: weak equation}.

\end{theorem}

\begin{proof}
See \cite[Theorem 2.2]{Friedrich-Kruzik-19}.
\end{proof}

 \begin{lemma}[Convexity and generalized geodesics]\label{th: convexity2}
Let $M >0$. Then there exist smooth increasing functions $\Theta^1,\Theta^2_M:[0,\infty) \to [0,\infty)$ satisfying   $\lim_{t \to 0}\Theta^1(t)/t =1$ and $\lim_{t \to 0}\Theta^2_M(t)/t =0$ such that for all $(u_0,v_0) \in {\mathscr{S}}$ with $\phi(u_0,v_0) \le M$  and all $(u_1,v_1) \in {\mathscr{S}}$ there holds   
\begin{align*}
(i)& \ \    \mathcal{D}\big((u_0,v_0), (u_s,v_s)\big)  \le  s \,  \Theta^1\big(\mathcal{D}\big((u_0,v_0), (u_1,v_1)\big)\big), \\
(ii) & \ \ \phi(u_s,v_s)  \le (1-s) \phi(u_0,v_0) + s\phi(u_1,v_1) +s \, \Theta^2_M\big(\mathcal{D}\big((u_0,v_0), (u_1,v_1)\big)\big),  
\end{align*}
where $u_s := (1-s) u_0 + su_1$ and  $v_s := (1-s) v_0 + sv_1$,  $s \in [0,1]$. 
\end{lemma}

\begin{proof}
See \cite[Lemma 4.8]{Friedrich-Kruzik-19}.
\end{proof}

 \begin{lemma}[Representation of energy and dissipation]\label{lemma: represi}
Let $\Omega = S \times (-\frac{1}{2},\frac{1}{2})$.  For $(u,v) \in \mathscr{S}$ we define for brevity
\begin{align}\label{eq: Gdef}
G(u,v)(x',x_3) =  e(u)(x') + \frac{1}{2} \nabla v(x') \otimes \nabla v(x')  -  x_3\nabla^2 v(x') \  \text{ for } x=(x',x_3) \in \Omega.
\end{align} 
Then $\phi$ and $\mathcal{D}$ can be represented as 
\begin{align}\label{eq: G and phi}
(i) & \ \ {\phi}(u,v) =  \int_\Omega \frac{1}{2}Q^2_W(G(u,v)),\notag \\
(ii) & \ \   {{\mathcal D}}\big( (u_1,v_1), (u_2,v_2)\big) = \Big( \int_\Omega Q^2_D\big(G(u_1,v_1) - G(u_2,v_2)\big) \Big)^{1/2}. 
\end{align}

\end{lemma}
\begin{proof}
See \cite[Remark 5.4]{Friedrich-Kruzik-19}
\end{proof}

We are now in a position to prove Theorem \ref{maintheorem}.  

\begin{proof}[Proof of Theorem \ref{maintheorem}]

\BBB First, \EEE note that $\phi$ and $\mathcal{D}$ are continuous with respect to $\rho$, see Lemma \ref{th: metric space-lin}(iii),(iv).  We now check that \eqref{basic assumptions2-new}-\eqref{eq: goedesi convexi} hold. First, \eqref{basic assumptions2-new} follows from the choice of $\sigma$,  Lemma \ref{th: metric space-lin}(ii), and a compactness argument. 
  
  Recall \eqref{eq: Gdef}. Given a sequence $(u_k,v_k)_k$ with $(u_k,v_k) \stackrel{\sigma}{\to} (u,v)$, we observe $G(u_k,v_k)\rightharpoonup G(u,v)$ weakly in $L^2(\Omega;\R^{2\times 2})$ since $W^{2,2} \subset\subset  \BBB W^{1,4} \EEE $ in dimension two. Then property  \eqref{compatibility-new} follows from \eqref{eq: G and phi}  and the fact that  $Q^2_W$ and $Q^2_D$  are positive semi-definite, \BBB see below \eqref{eq: order4}. \EEE By the definition of $\mathscr{S}_k$ we get  \eqref{eq: strong convergence}.  To see \eqref {eq: goedesi convexi}, we use Lemma \ref{th: convexity2} and the fact that, if $(u_0,v_0),(u_1,v_1) \in \mathscr{S}_k$,  the convex combinations also lie in $\mathscr{S}_k$ due to the convexity of the sets $\mathscr{S}_k$.   

It remains to prove \eqref{eq: mutual recovery}. To this end, consider a sequence $(u_k,v_k)$ with $(u_k,v_k) \stackrel{\sigma}{\to} (u,v)$, and recall that  $G(u_k,v_k)\rightharpoonup G(u,v)$ weakly in $L^2(\Omega;\R^{2\times 2})$.    Suppose that also $(\bar{u},\bar{v}) \in \mathscr{S}$ is given. We define $\bar{u}_k = u_k  + \bar{u}- u$ and   $\bar{v}_k = v_k + \bar{v} - v$. Then by \eqref{eq: Gdef} and an elementary expansion we get
\begin{align*}
&G(u_k,v_k) -  G(\bar{u}_k,\bar{v}_k) \\
& =  e(u - \bar{u}) - x_3\nabla^2 (v- \bar{v}) -  \frac{1}{2} \nabla (v-\bar{v}) \otimes \nabla (v-\bar{v})  + \text{sym}(\nabla v_k \otimes \nabla (v- \bar{v})). 
\end{align*}
Since $v_k \to v$ strongly in $W^{1,4}(S)$, we get that $G(u_k,v_k) -  G(\bar{u}_k,\bar{v}_k)$ converges strongly in $L^2(\Omega;\R^{2 \times 2})$ to
\begin{align*}
e(u - \bar{u}) - x_3\nabla^2 (v- \bar{v}) -  \frac{1}{2} \nabla (v-\bar{v}) \otimes \nabla (v-\bar{v})  + \text{sym}(\nabla v \otimes \nabla (v- \bar{v})), 
\end{align*}
i.e.,   $G(u_k,v_k) -  G(\bar{u}_k,\bar{v}_k)$ converges strongly in $L^2(\Omega;\R^{2 \times 2})$ to $G(u,v) -  G(\bar{u},\bar{v})$. In view of \eqref{eq: G and phi}(ii), this implies 
\begin{align}\label{eq: Dconv}
{{\mathcal D}}\big( (u_k,v_k), (\bar{u}_k,\bar{v}_k)\big) \to  {{\mathcal D}}\big( (u,v), (\bar{u},\bar{v})\big).
\end{align}
Moreover, by an elementary expansion and    \eqref{eq: G and phi}(i) we get   
\begin{align*}
{\phi}(u_k,v_k) -  {\phi}(\bar{u}_k,\bar{v}_k) & =  \int_\Omega \frac{1}{2} \big(Q^2_W(G(u_k,v_k)) - Q^2_W(G( \bar{u}_k,\bar{v}_k))  \big) \\
& = -\int_\Omega \C_W^2[G(u_k,v_k), G(\bar{u}_k,\bar{v}_k)  - G(u_k,v_k) ]  \\ 
& \ \ \ - \int_\Omega\frac{1}{2}Q^2_W \big(G(\bar{u}_k,\bar{v}_k)  - G(u_k,v_k) \big).
\end{align*}
Since  $G(u_k,v_k) -  G(\bar{u}_k,\bar{v}_k)$ converges strongly to $G(u,v) -  G(\bar{u},\bar{v})$  and $G(u_k,v_k)\rightharpoonup G(u,v)$ weakly in $L^2(\Omega;\R^{2 \times 2})$, we   get ${\phi}(u_k,v_k) -  {\phi}(\bar{u}_k,\bar{v}_k) \to {\phi}(u,v) -  {\phi}(\bar{u},\bar{v})$. This \BBB along with \eqref{eq: Dconv} \EEE  shows that  \eqref{eq: goedesi convexi} holds.

Having checked \eqref{basic assumptions2-new}-\eqref{eq: goedesi convexi}, we can now apply  Theorem \ref{th:abstract convergence-numerics}. This yields the existence of  time-discrete solutions and of a curve of maximal slope such that  convergence of time-discrete solutions holds with respect to $\sigma$. 
The fact that the  curve is a weak solution to  \eqref{eq: equation-simp}  follows from Theorem \ref{maintheorem2}. It remains to prove that the convergence of time-discrete solutions holds with respect to the strong topology $\rho$. 

To confirm the latter property, we use the principle that weak convergence together with energy convergence induces strong convergence. More specifically, given $(u_k,v_k) \stackrel{\sigma}{\to} (u,v)$ and $\phi(u_k,v_k) \to \phi(u,v) $, we argue as follows. Since $G(u_k,v_k) \rightharpoonup G(u,v) $ weakly in $L^2(\Omega;\R^{2 \times 2})$, we get by \eqref{eq: G and phi}(i) that
\begin{align*}
\int_\Omega\frac{1}{2}Q^2_W \big(G(u_k,v_k) - G(u,v)   \big) & = \int_\Omega \C_W^2[G(u,v), G(u,v)  - G(u_k,v_k) ] \\
& \ \ \ + {\phi}(u_k,v_k) -  {\phi}(u,v)  \to 0.
\end{align*}
Since $Q^2_W$ is positive definite on $\R_{\rm sym}^{2\times 2}$, \BBB see below \eqref{eq: order4}, \EEE we get ${\rm sym}(G(u_k,v_k))\to {\rm sym}(G(u,v))$ strongly in $L^2(\Omega;\R^{2 \times 2}_{\rm sym})$. Then, \BBB in view of \eqref{eq: Gdef}, \EEE using Poincar\'e's and  Korn's inequality, together with zero boundary conditions, it is elementary to check that $u_k \to u$ in $W^{1,2}(S;\R^2)$ and $v_k \to v$ in $W^{2,2}(S)$, i.e.,   $(u_k,v_k)  \stackrel{\rho}{\to} (u,v)$. 
\end{proof}

\begin{figure}[h]
 \begin{minipage}[c]{.48\textwidth}
        \centering
        \includegraphics[width=\textwidth]{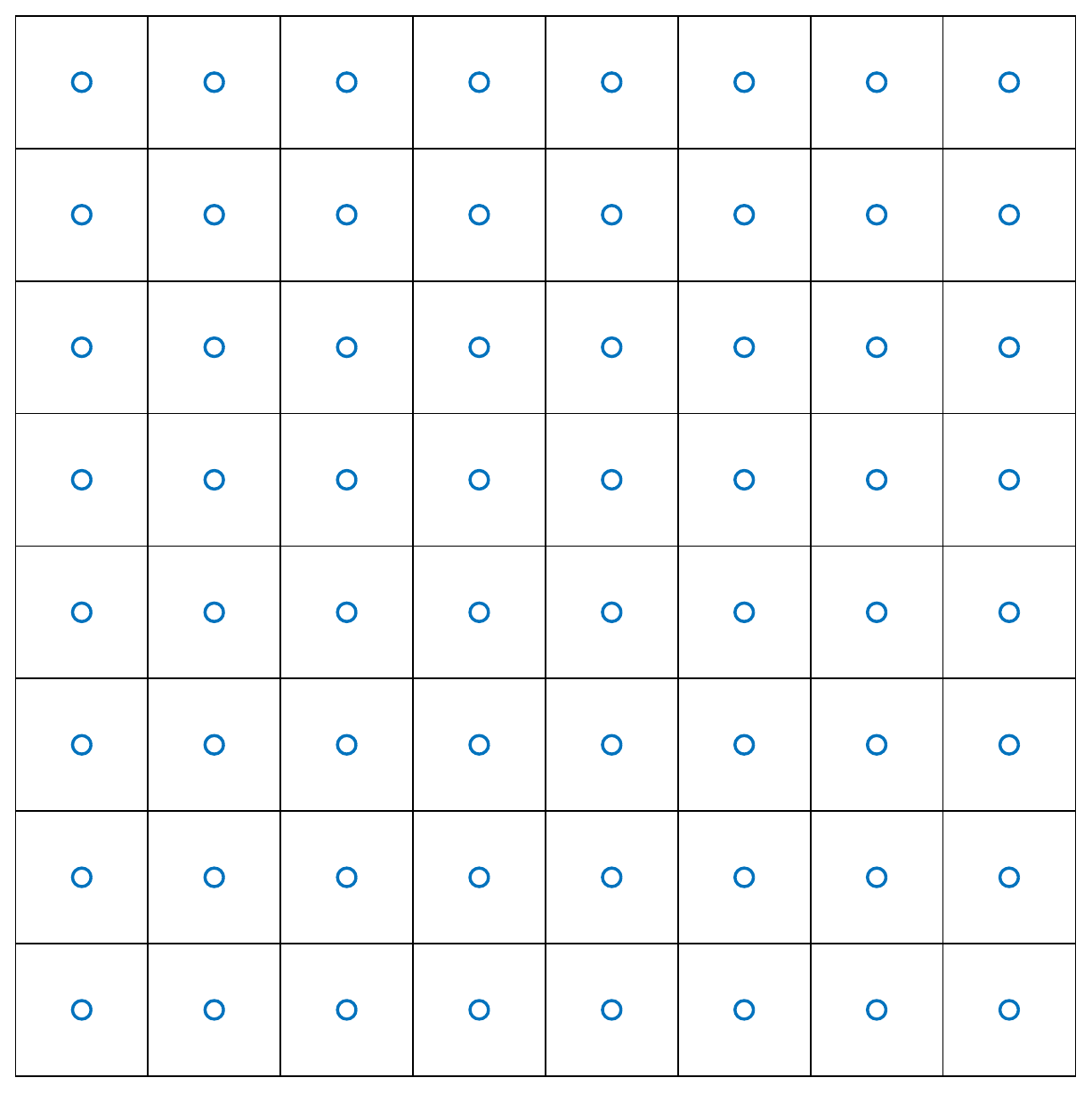}
 \end{minipage}  
  \hspace{0.01\textwidth}
  \begin{minipage}[c]{.48\textwidth}
        \centering
        \includegraphics[width=\textwidth]{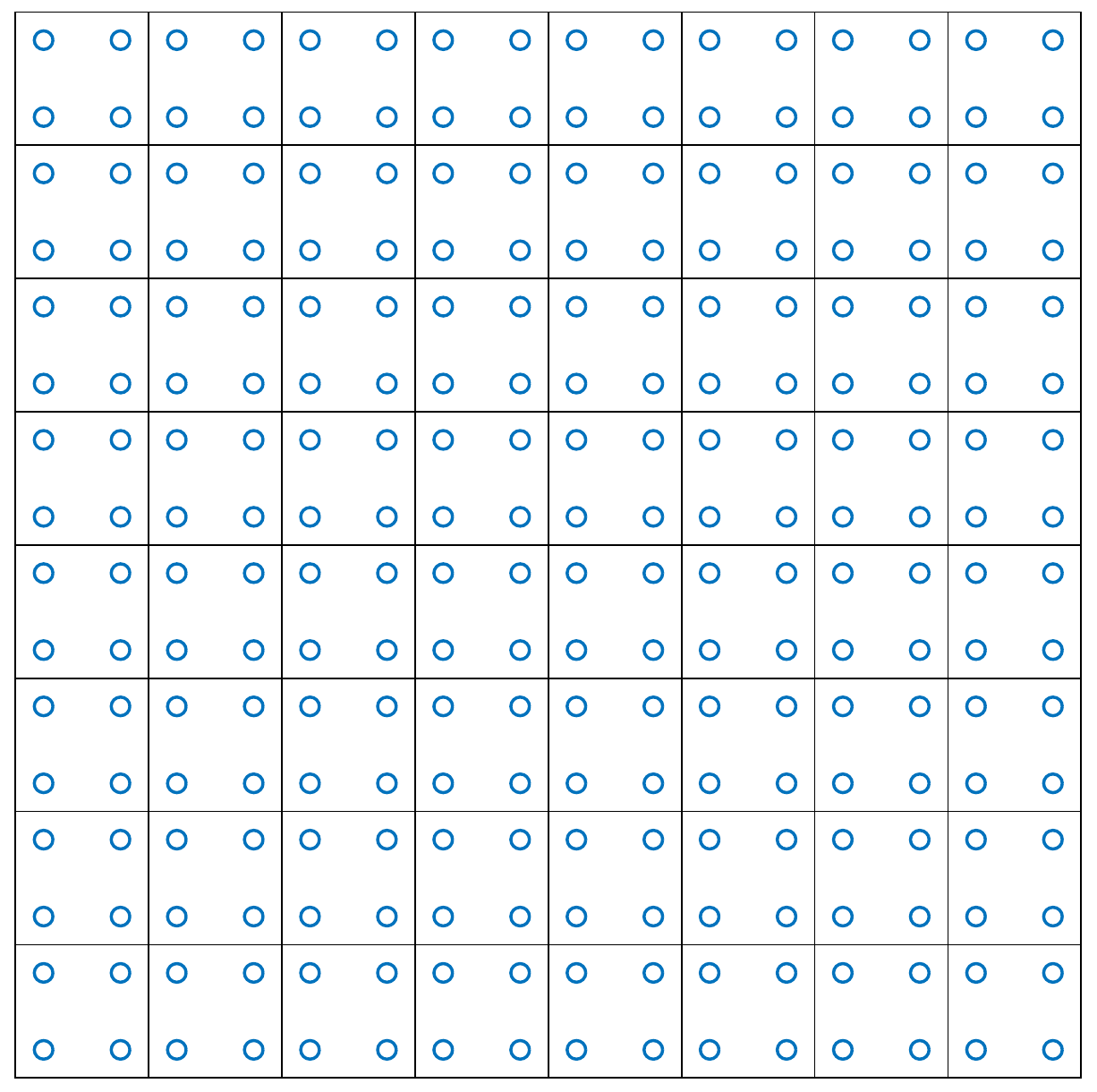}
 \end{minipage}  
  \caption{A rectangular mesh with Gauss integration points:  midpoints are used for evaluation of Q1 elements (left) and four points for evaluation of Bogner-Fox-Schmit elements  (right). 
}
   \label{fig:setup}
\end{figure}

\section{Numerical experiments}\label{numexp}
In this section we describe two numerical experiments on a homogeneous and isotropic viscoelastic plate.    Our computational strategy relies on \BBB a \EEE sequence of minimization problems
based on \eqref{eq: phi0} and \eqref{eq: D,D0-2}.  Take a time horizon $T>0$  and a time step $\tau>0$  such that $ n_{max} :=T/\tau\in\N$.  
Having an initial condition $(u_{0},v_{0})\in\mathscr{S}_k$ \BBB in a finite element space $\mathscr{S}_k$ detailed below, \EEE we find, for $1 \leq n \leq n_{max}$,  a solution  $(u_{n},v_{n})\in\mathscr{S}_k$  of the following problem 
\begin{align}\label{eq: ds-new1NNN}
&\text{minimize } \phi(u,v) +\frac{1}{2\tau}\mathcal{D}^2( (u_{n-1},v_{n-1}),(u,v))\nonumber\\
& \text{subject to } (u,v)\in {\mathscr{S}_k},
\end{align}
where $\phi$ and $\mathcal{D}$ are defined in \eqref{eq: phi0} and \eqref{eq: D,D0-2}, respectively. 
As $Q^2_W$ we take an isotropic material, i.e., 
$$Q^2_W(G ):= \lambda {\rm tr}^2(G)+ 2\mu|G|^2,$$
and 
$$Q^2_D(G):= 4c|G|^2\ , \ c>0\ $$
  for every  symmetric $G\in \R^{2\times 2}_{\rm sym}$. The constants $\lambda, \mu$ are Lam\'e constants  and $c>0$ represents a viscosity parameter.

A finite element method (FEM) is applied for the numerical approximation of the in-plane and  out-of-plane displacement fields $u$ and $v$. This space corresponds to $\mathscr{S}_k$, for some $k \in \N$ large enough, as considered in Section \ref{sec:FEM}. We assume a uniform rectangular mesh in 2D discretizing a square domain

$$S = (-1 ,1 ) \times (-1, 1) $$ into square elements with the edge of the length $h=1/k$ 
 and further approximate:
\begin{enumerate}
\item a vector function $u = (u_1, u_2)$ by $\text{Q1}_k$ elements (elementwise bilinear and globally continuous)  in each component $u_1$ and $u_2$,
\item a scalar function $v$  by the Bogner-Fox-Schmit ($\text{BFS}_k$) rectangular elements \cite{BoFoSc65}, i.e., a bi-cubic Hermite elements,  that provide globally  $C^1$ approximations. 
\end{enumerate}
\BBB We \EEE define $\mathscr{S}_k= \{(u,v)\in \text{Q1}_k\times\text{BFS}_k\}$.  The $\rho$-denseness of $\bigcup_k\mathscr{S}_k$ in $\mathscr{S}$ follows from the properties of the finite-element interpolants, see \cite[Thms.~3.2.3 and 6.1.7]{Ciarlet-FEM} or \cite{BS}.   A rectangular mesh with
81 nodes and 64 rectangles is given in Figure \ref{fig:setup} together with Gauss integration points used in quadrature formulas. 

\begin{figure}
    \centering
    \begin{subfigure}[b]{0.49\textwidth}
    \centering
        \includegraphics[width=0.85\textwidth]
        {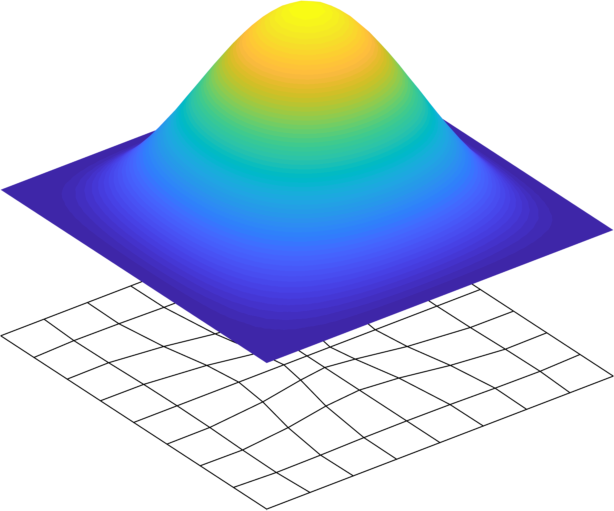}
        \caption{t=1}
    \end{subfigure}
    \begin{subfigure}[b]{0.49\textwidth}
    \centering
        \includegraphics[width=0.85\textwidth]
          {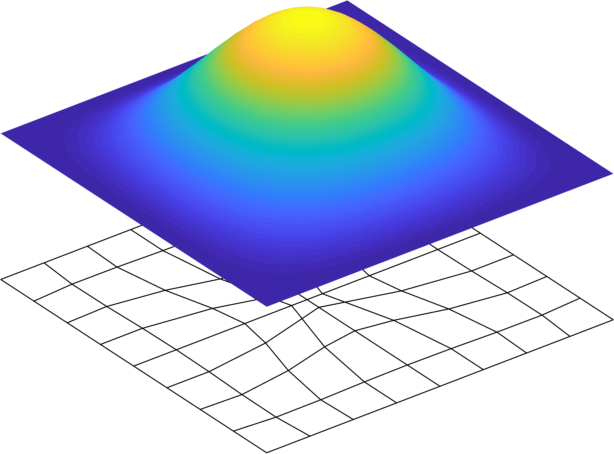}
        \caption{t=2}
    \end{subfigure}
     \begin{subfigure}[b]{0.49\textwidth}
    \centering
        \includegraphics[width=0.85\textwidth]
          {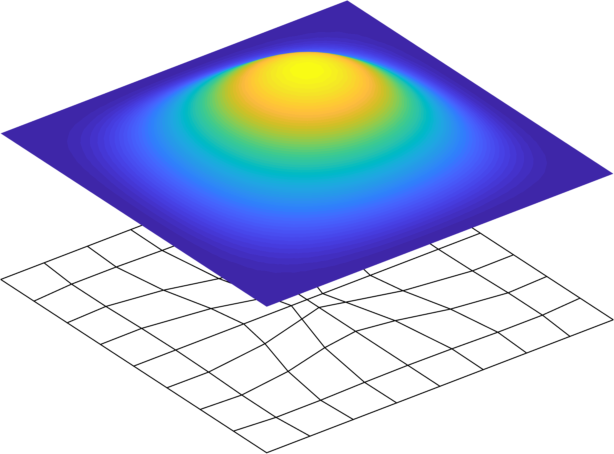}
        \caption{t=3}
    \end{subfigure}
    \begin{subfigure}[b]{0.49\textwidth}
    \centering
        \includegraphics[width=0.85\textwidth]
          {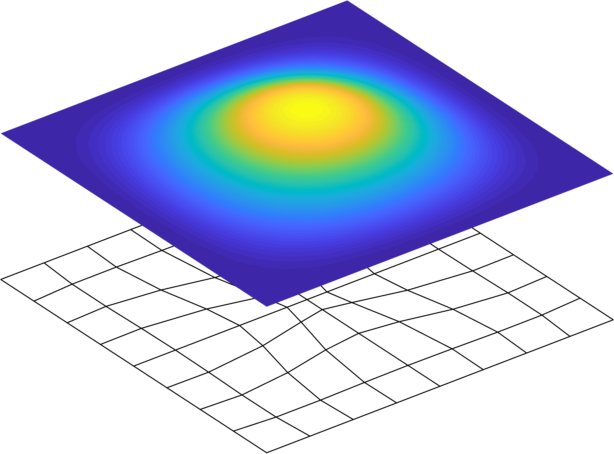}
        \caption{t=4}
    \end{subfigure}
    \begin{subfigure}[b]{0.49\textwidth}
    \centering
        \includegraphics[width=0.85\textwidth]
          {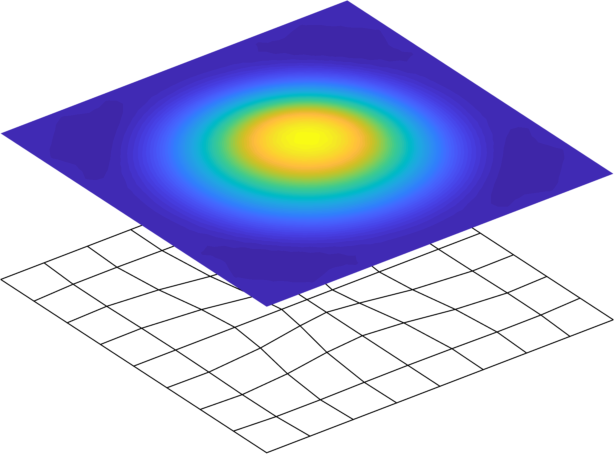}
        \caption{t=5}
    \end{subfigure}
    \begin{subfigure}[b]{0.49\textwidth}
    \centering
        \includegraphics[width=0.85\textwidth]
          {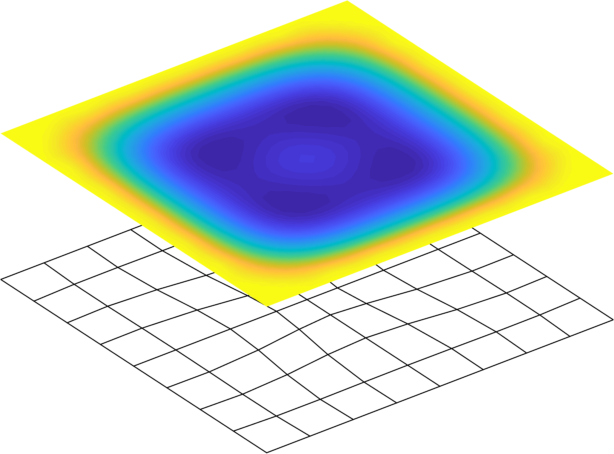}
        \caption{t=6}
    \end{subfigure}
     \begin{subfigure}[b]{0.49\textwidth}
    \centering
        \includegraphics[width=0.85\textwidth]
          {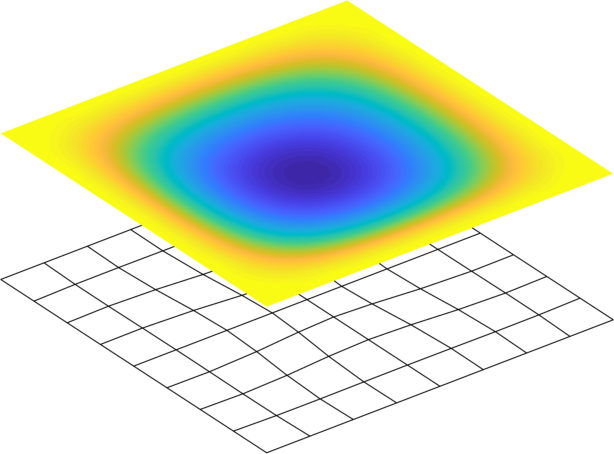}
        \caption{t=7}
    \end{subfigure}
    \begin{subfigure}[b]{0.49\textwidth}
    \centering
        \includegraphics[width=0.85\textwidth]
          {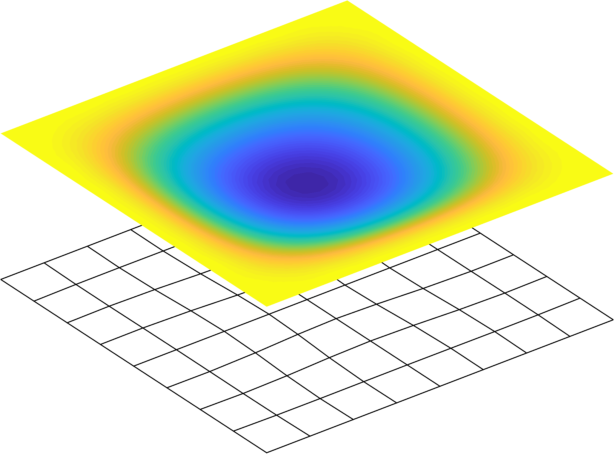}
        \caption{t=8}
    \end{subfigure}
 \caption{
    Time sequence of energy minimizers $(u,v)$ in Benchmark I. To emphasize the mesh deformation, the nodes displacement is magnified by a factor of 4. }
   \label{fig:benchmark1_time_sequence}
\end{figure}

\begin{figure}[htbp]
    \centering
    \begin{subfigure}[b]{0.49\textwidth}
    \centering
        \includegraphics[width=0.78\textwidth]
        {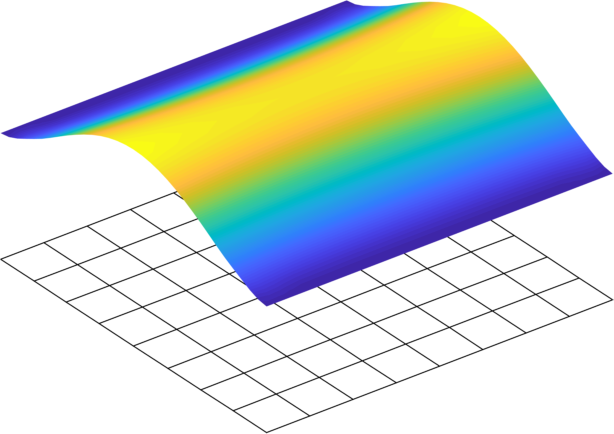}
        \caption{t=1}
    \end{subfigure}
    \begin{subfigure}[b]{0.49\textwidth}
    \centering
        \includegraphics[width=0.78\textwidth] {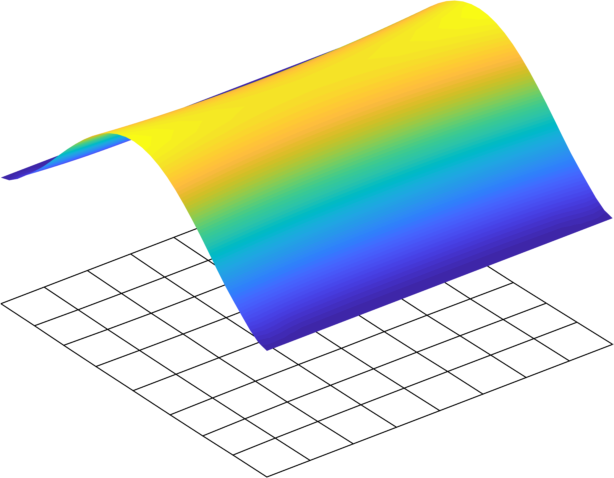}
        \caption{t=2}
    \end{subfigure}
     \begin{subfigure}[b]{0.49\textwidth}
    \centering
        \includegraphics[width=0.78\textwidth] {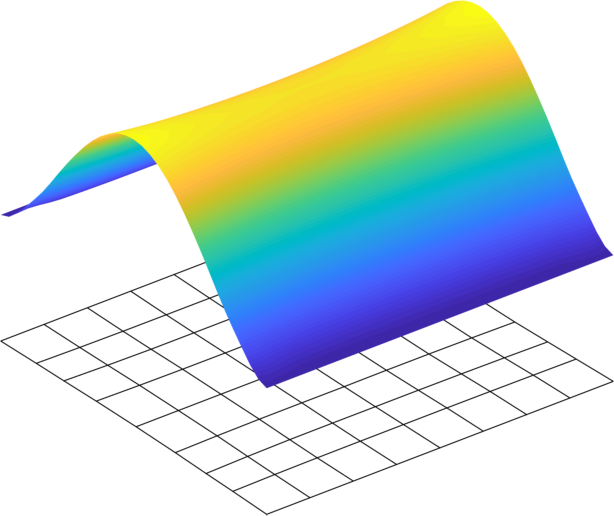}
        \caption{t=3}
    \end{subfigure}
    \begin{subfigure}[b]{0.49\textwidth}
    \centering
        \includegraphics[width=0.78\textwidth] {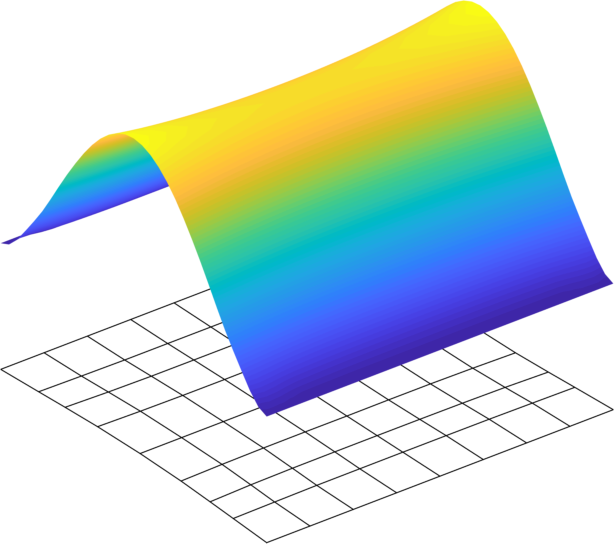}
        \caption{t=4}
    \end{subfigure}
    \begin{subfigure}[b]{0.49\textwidth}
    \centering
        \includegraphics[width=0.78\textwidth] {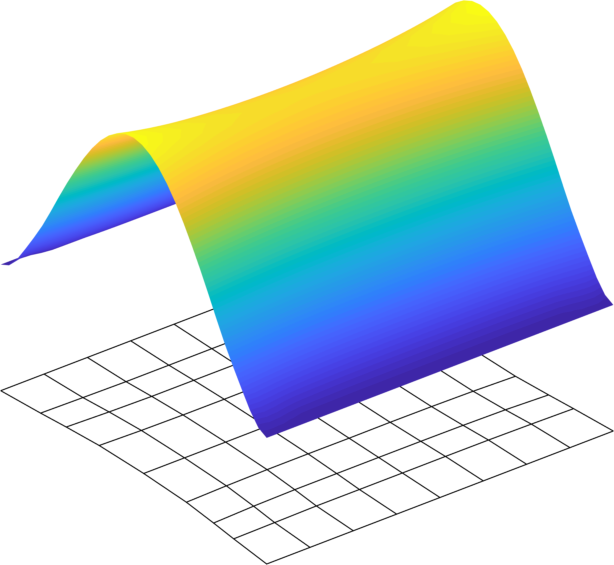}
        \caption{t=5}
    \end{subfigure}
    \begin{subfigure}[b]{0.49\textwidth}
    \centering
        \includegraphics[width=0.78\textwidth] {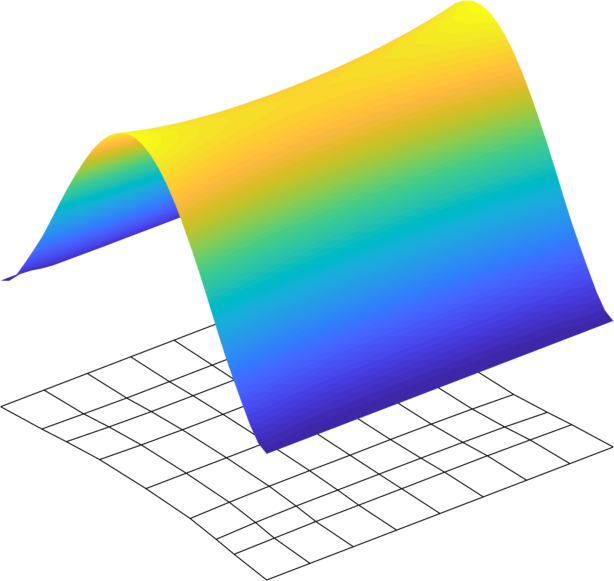}
        \caption{t=6}
    \end{subfigure}
    \begin{subfigure}[b]{0.49\textwidth}
   \centering
        \includegraphics[width=0.78\textwidth] {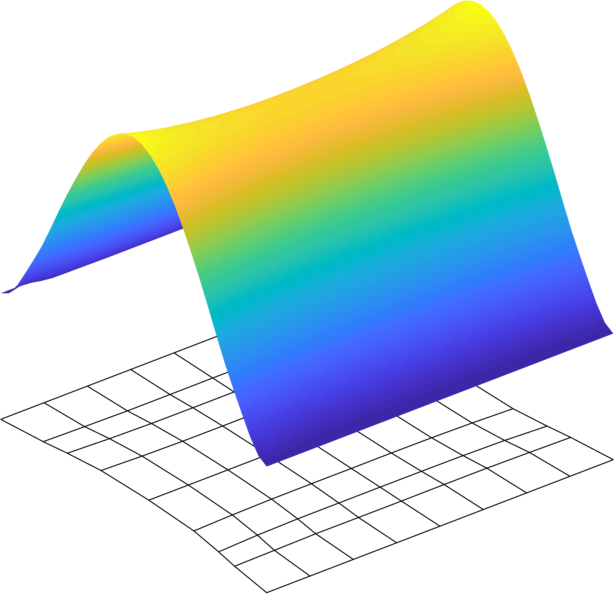}
        \caption{t=7}
    \end{subfigure}
    \begin{subfigure}[b]{0.49\textwidth}
    \centering
        \includegraphics[width=0.78\textwidth] {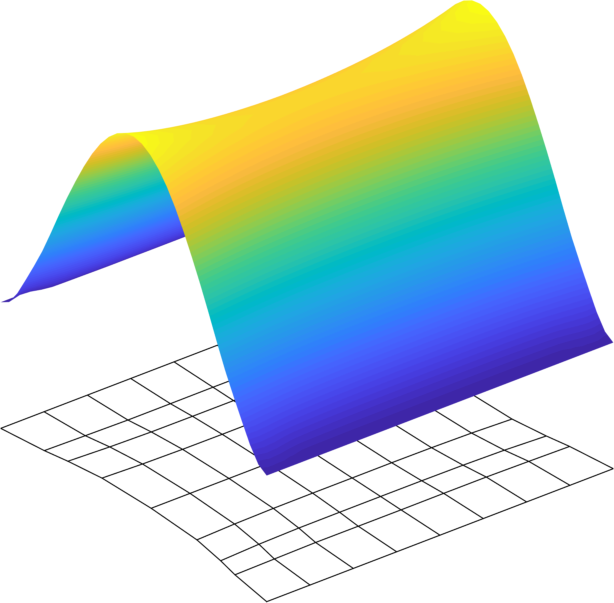}
        \caption{t=8}
    \end{subfigure}
 \caption{
    Time sequence of energy minimizers $(u,v)$ in Benchmark II. To emphasize the mesh deformation, the nodes displacement is magnified by a factor of 7. }
   \label{fig:benchmark2_time_sequence}
\end{figure}

\subsection{Benchmark I}
We consider a time sequence of minimization problems with the time step $\tau=1$, Lam\'e parameters $\lambda = \mu = 1e3$, the viscosity parameter $c=3e3$, and the constant volume force $f=-1e3$.  The initial condition is  given by 
\begin{equation}
u_0=(0, 0), \quad  v_0= (1-x_1^2)^2(1-x_2^2)^2 \quad \mbox{for } x=(x_1, x_2) \in S, 
   \label{benchI:init}
\end{equation}
and the boundary condition for $t \geq 0$  by
\begin{equation*}
u(t, x)=(0,0), \quad v(t,x)=0, \quad \nabla v(t,x)=(0,0) \qquad \mbox{for } x \in \partial S. 
\end{equation*}
Although the choice of parameters above is not physically relevant, the meaning of this benchmark is clear: The initial condition $(u_0, v_0)$ is not equilibrated and therefore after some time (the speed of this transition is driven by the viscosity constant $c$)  the energy stabilizes at its equilibrium given by the volume force oriented in the gravity direction. The  first 8 minimizers are displayed in Figure \ref{fig:benchmark1_time_sequence}.  The scalar field $v$ is displayed as a vertical plate deformation, whereas the vector displacement field $u$ is displayed as a deformed mesh. 

\subsection{Benchmark II}
The theoretical part of this paper covers  boundary conditions defined on the full boundary of  $S$  only.  The computer simulations, however, are  possible also for boundary conditions given on a part of the boundary $\partial S $. 
We consider a time sequence of minimization problems with the time step $\tau=1$, Lam\'e parameters $\lambda = \mu = 1e3$, the viscosity parameter $c=3e3$, and the constant volume force $f=1e2$.  The initial condition is  given by 
\begin{equation*}
u_0=(0, 0), \quad  v_0= 0 \quad \mbox{for } x \in S, 
\end{equation*}
and the boundary condition for $t \geq 0$  by
\begin{equation*}
u(t, x)=(0,0), \quad v(t,x)=0, \quad \nabla v(t,x)=(0,0) \quad \mbox{for } x \in (-1,1) \times \{-1, 1 \},
\end{equation*}
so boundary conditions are given on the lower and upper parts of the domain boundary only.
The  first 8 minimizers are displayed in Figure \ref{fig:benchmark2_time_sequence}. 

\begin{figure}
    \centering
    \begin{subfigure}[b]{0.49\textwidth}
    \centering
        \includegraphics[width=0.855\textwidth]{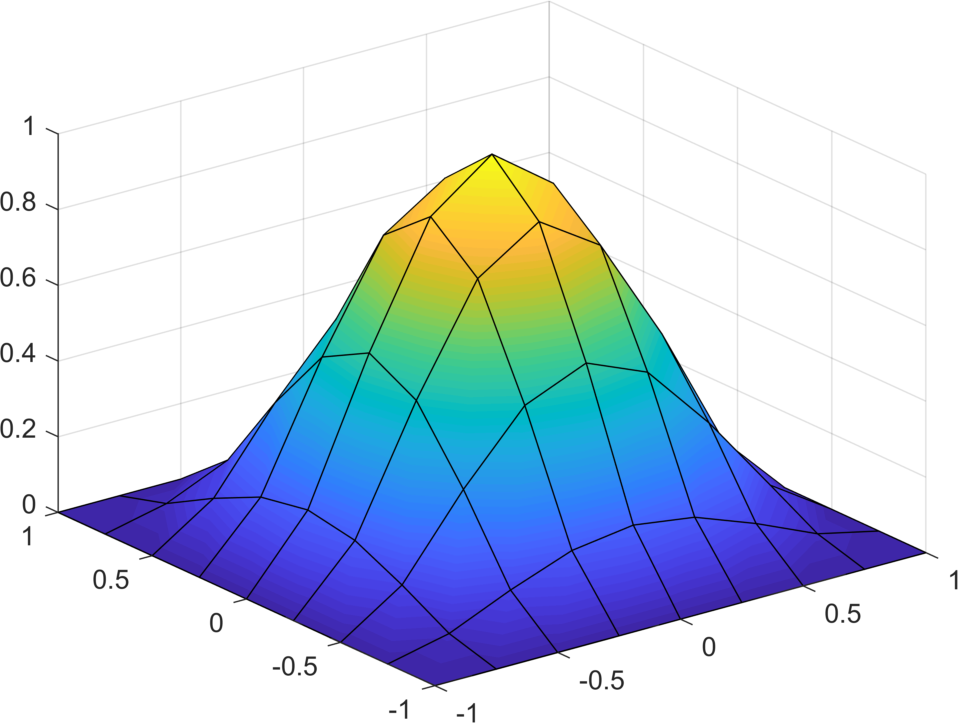}
        \caption{$v_0$}
    \end{subfigure}
    \begin{subfigure}[b]{0.49\textwidth}
    \centering
        \includegraphics[width=0.85\textwidth]
        {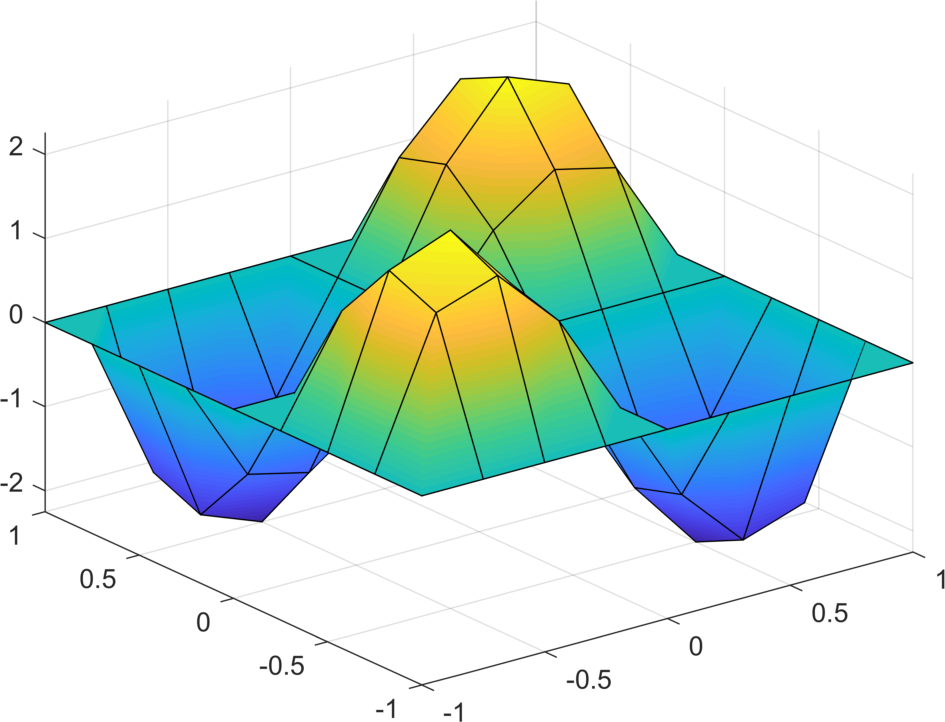}
       \caption{$\frac{\partial^2 v_0}{\partial x_1 \partial x_2}$}
    \end{subfigure}
     \begin{subfigure}[b]{0.49\textwidth}
    \centering
        \includegraphics[width=0.85\textwidth]{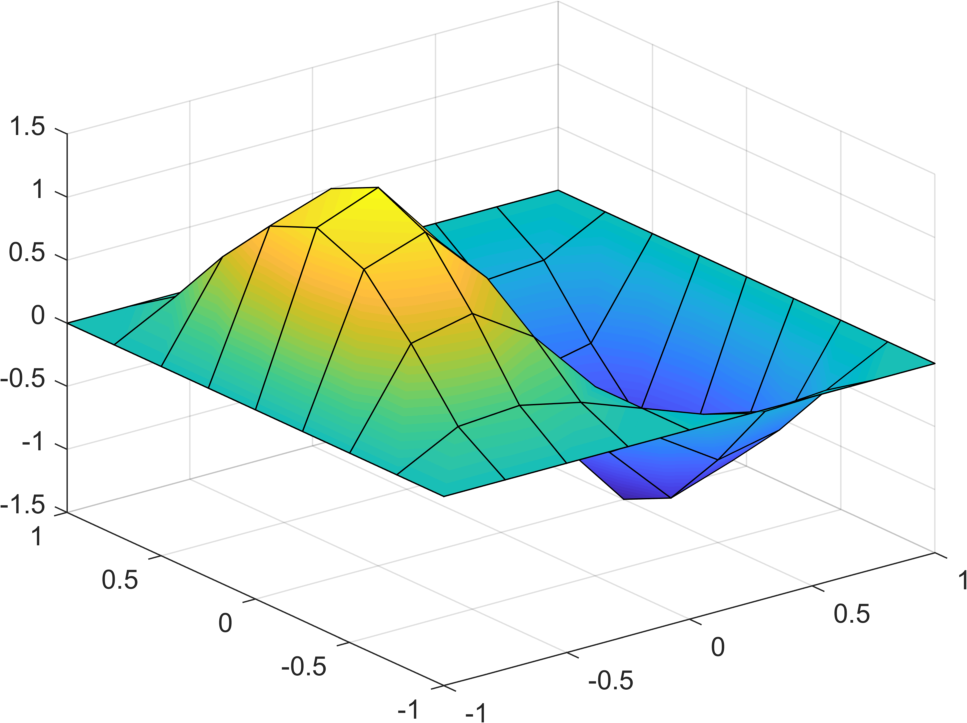}
 \caption{$\frac{\partial v_0}{\partial x_1}$}        
    \end{subfigure}
    \begin{subfigure}[b]{0.49\textwidth}
    \centering
        \includegraphics[width=0.85\textwidth]
        {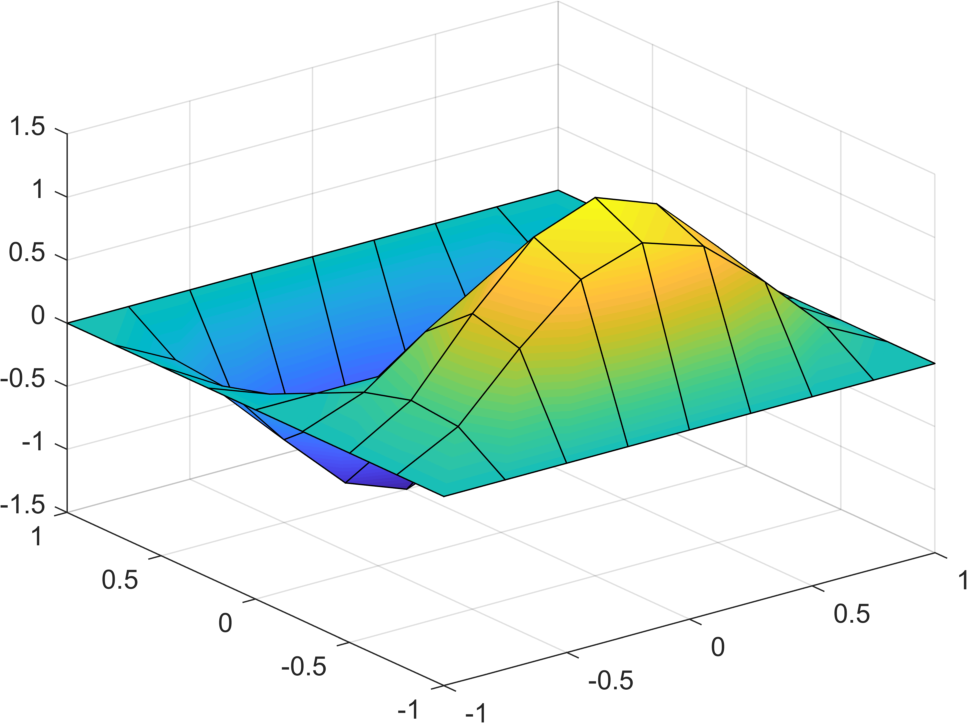}
         \caption{$\frac{\partial v_0}{\partial x_2}$ }
    \end{subfigure}
 \caption{Example of $C^1$ approximation by the  Bogner-Fox-Schmit (BFS) rectangular elements:  the initial function $v_0$ from Benchmark 1 is represented by its value, its gradient and the second mixed derivative {\cm in all mesh nodes.  
 }}
  \label{fig:example_v}
\end{figure}

\subsection{Implementation details}
Our Matlab implementation is based on former vectorized codes of \cite{ AnVa15,HaVa14,RaVa13} that allow for a fast assembly of various finite element matrices. 
{\cm
The code is available at
\begin{center}
\url{https://www.mathworks.com/matlabcentral/fileexchange/72991} 
\end{center}}
{\flushleft
for} download. It includes an own implementation of the Bogner-Fox-Schmit (BFS) rectangular elements for a uniformly refined rectangular mesh, where all rectangular elements are for simplicity of the same size $hx_1 \times hx_2$ (in our computations $hx_1 = hx_2 = h$).
The basis functions on each rectangle are based on bicubic polynomials, i.e., tensor products of 4 cubic (Hermite) polynomials. They have 16 degrees of freedom with 4 degrees in each of its 4 corner nodes approximating: a function value, its gradient (two components), and the second mixed derivative. 
Therefore, the initial function $v_0$ must have all these fields available in our simulations. Figure \ref{fig:example_v} depicts $v_0$ from \eqref{benchI:init} represented in terms of BFS elements. We recall that BFS elements were also successfully tested in \cite{KrVa18} {\cm and their implementation was explained in detail in \cite{Valdman}}.



\section*{Acknowledgments} The research of MF was supported  by the Deutsche Forschungsgemeinschaft (DFG, German Research Foundation) under Germany's Excellence Strategy EXC 2044 -390685587, Mathematics M\"unster: ``Dynamics--Geometry--Structure''.  MK and JV acknowledge the support by GA\v{C}R project  17-04301S.

\medskip

\medskip

\end{document}